\documentclass[11pt]{article}

\usepackage{amsmath}
\usepackage{amsfonts}
\usepackage{amssymb}
\usepackage{amsthm}
\usepackage{enumitem}
\usepackage{mathtools}
\usepackage{setspace}
\usepackage{etoolbox}
\usepackage{color}
\usepackage{graphicx}
\usepackage{verbatim}
\usepackage[normalem]{ulem}

\newtheoremstyle{mystyle}
{11pt}				% above space 
{11pt}				% below space 
{\itshape}					% body font 
{}					% indent amount 
{\bfseries}			% head font 
{}					% post head punctuation 
{5.5pt}				% space between head and body
{}					% head spec

\newtheoremstyle{mystyle2}
{11pt}				% above space 
{11pt}				% below space 
{}					% body font 
{}					% indent amount 
{\bfseries}			% head font 
{}					% post head punctuation 
{5.5pt}				% space between head and body
{}					% head spec

\theoremstyle{mystyle}

\newtheorem{theorem}{Theorem}[section]
\newtheorem{definition}[theorem]{Definition}

\newtheorem{proposition}[theorem]{Proposition}
\newtheorem{corollary}[theorem]{Corollary}
\theoremstyle{mystyle2}
\newtheorem{example}[theorem]{Example}

\renewenvironment{proof}[1][Proof.]{\vspace{-11pt} \begin{trivlist}
		\item[\hskip \labelsep {\bfseries #1}]}{\qed \end{trivlist}}

\setitemize{noitemsep, topsep=5.5pt, parsep=5.5pt, partopsep=0pt}

\allowdisplaybreaks
\newcommand{\gap}{\vspace{11pt}}
\newcommand{\smallgap}{\vspace{5.5pt}}

\newcommand{\diag}{\operatorname{diag}}
\newcommand{\tr}{\operatorname{tr}}
\newcommand{\conv}{\operatorname{conv}}
\newcommand{\spn}{\operatorname{span}}
\newcommand{\Aut}{\operatorname{Aut}}
\newcommand{\cone}{\operatorname{cone}}

\newcommand{\R}{\mathcal{R}}
\newcommand{\Rn}{\mathcal{R}^n}

\newcommand{\Sn}{\mathcal{S}^n}

\newcommand{\Hn}{\mathcal{H}^n}
\newcommand{\V}{{\cal V}}

\newcommand{\W}{{\cal W}}
\newcommand{\A}{{\cal A}}
\newcommand{\G}{{\cal G}}
\newcommand{\h}{{\cal H}}
\newcommand{\lc}{\lambda(c)}
\newcommand{\ly}{\lambda(y)}
\newcommand{\lx}{\lambda(x)}
\newcommand{\lu}{\lambda(u)}
\newcommand{\lv}{\lambda(v)}

\newcommand{\abs}[1]{\left\vert #1 \right\vert}
\newcommand{\norm}[1]{\left\Vert #1 \right\Vert}
\newcommand{\ip}[2]{\left< #1, #2 \right>}
\newcommand{\starprec}{\mathrel{\begin{aligned}[t] \prec \\[-3.2ex] {\scriptstyle *\;} \end{aligned}}}

\title{\bf Commutativity, majorization, and reduction in \\
	Fan-Theobald-von Neumann systems 
}
\author{
	M. Seetharama Gowda\\
	Department of Mathematics and Statistics\\
	University of Maryland, Baltimore County\\
	Baltimore, Maryland 21250, USA\\
	gowda@umbc.edu\\
	and\\
	Juyoung Jeong \\
	Applied Algebra and Optimization Research Center \\
	Sungkyunkwan University \\
	2066 Seobu-ro, Suwon 16419, Republic of Korea \\
	jjycjn@skku.edu}
\date{\today}

\begin{document}

\maketitle

\begin{abstract}
	A Fan-Theobald-von Neumann system \cite{gowda-ftvn} is a triple $(\V,\W,\lambda)$, where $\V$ and $\W$ are real inner product spaces and $\lambda:\V \to \W$ is a norm-preserving map satisfying a Fan-Theobald-von Neumann type inequality together with a condition for equality. Examples include Euclidean Jordan algebras, systems induced by certain hyperbolic polynomials, and normal decompositions systems (Eaton triples). In \cite{gowda-ftvn}, we presented some basic properties of such systems and described results on optimization problems dealing with certain combinations of linear/distance and spectral functions. 
	We also introduced the concept of commutativity via the equality in the Fan-Theobald-von Neumann type inequality.
	In the present paper, we elaborate on the concept of commutativity and introduce/study automorphisms, majorization, and reduction in Fan-Theobald-von Neumann systems.
\end{abstract}

\vspace{1cm}
\noindent{\bf Key Words}: Fan-Theobald-von Neumann system, Euclidean Jordan algebra, normal decomposition system, Eaton triple, hyperbolic polynomial, spectral set, eigenvalue map, strong operator commutativity 
\\

\noindent{\bf AMS Subject Classification:} 
15A27, 17C20, 46N10, 52A41, 90C25, 90C33.

%%%%%%%%%%%%%%%%%%%%%%%%%%%%%%%%%%%%%%%%%%%%%%%%%%%%%%%%%%%%%%%%%%%%%%%%%%%
\section{Introduction}

A Fan-Theobald-von Neumann system (FTvN system, for short), introduced in \cite{gowda-ftvn}, is a triple $(\V,\W,\lambda)$, where $\V$ and $\W$ are real inner product spaces and $\lambda:\V \to \W$ is a norm-preserving map satisfying the property
\begin{equation} \label{intro ftvn}
	\max \Big\{ \ip{c}{x} : \, x \in [u] \Big\} = \ip{\lc}{\lu} \quad (c,u\in \V),
\end{equation}
with $[u]:=\{x\in \V: \lx=\lu\}$ denoting the so-called $\lambda$-orbit of $u\in \V$. This property is a combination of an inequality and a condition for equality, see Section 2 for details and an elaborated version. The inequality
\[ \ip{x}{y} \leq \ip{\lx}{\ly} \quad (x,y \in \V), \]
which comes from (\ref{intro ftvn}) will be called {\it Fan-Theobald-von Neumann inequality} and the equality
\[ \ip{x}{y} = \ip{\lx}{\ly} \]
defines the {\it commutativity} of $x$ and $y$ in this system. Examples of such systems include \cite{gowda-ftvn}: 

\begin{itemize}
	\item [$(a)$] The triple $(\V,\Rn,\lambda)$, where  $\V$ is a Euclidean Jordan algebra of rank $n$ carrying the trace inner product with $\lambda:\V \to \Rn$ denoting the eigenvalue map, 
	
	\item [$(b)$] The triple $(\V,\Rn,\lambda)$, where $\V$ is a finite dimensional real vector space and $p$ is a real homogeneous polynomial of degree $n$ that is hyperbolic with respect to a vector $e\in \V$, complete and isometric, with $\lx$ denoting the vector of roots of the univariate polynomial $t \mapsto p(te-x)$ written in the decreasing order, and 
	
	\item [$(c)$] The triple $(\V,\W,\gamma)$ where $(\V,\G,\gamma)$ is a normal decomposition system (in particular, an Eaton triple) and $\W := \spn(\gamma(\V))$.
\end{itemize}

The article \cite{gowda-ftvn} covered basic properties of FTvN systems and was primarily focused on ways of transforming certain optimization problems on $\V$ to problems on $\W$ and describing attainment in terms of commutativity. A subsequent article \cite{gowda-ftvn-commutation} dealt with  certain related commutation principles.  In the present article, we  introduce and study the concepts of center, unit element, automorphism, majorization, and reduced system. While some of these concepts are known in familiar settings, our goal here is to present them in the general framework of FTvN systems. We now  briefly define the terms  with  illustrations and results to come later.

\gap
Let $(\V,\W,\lambda)$ be a FTvN system.
\begin{itemize}
	\item The {\it center} of $(\V,\W,\lambda)$ consists of those elements of $\V$ that commute with every element of $\V$. A nonzero element $e\in \V$ is a {\it unit element} if the center consists of just (scalar) multiples of $e$. 
	
	\item An invertible linear transformation $A:\V \to \V$ is said to be an {\it automorphism} of $(\V,\W,\lambda)$ if $\lambda(Ax) = \lx$ for all $x \in \V$.
	
	\item Given two elements $x,y\in \V$, we say that {\it $x$ is majorized by $y$} and write $x\prec y$ if $x \in \conv\,[y]$, where `$\conv$' is an abbreviation for convex hull.
	
	\item A linear transformation $D:\V \to \V$ is called {\it doubly stochastic} if $Dx\prec x$ for all $x\in \V$.
	
	\item If there is a map $\mu:\W \to \W$ such that $(\W,\W,\mu)$ is a FTvN system with $\mu\circ \lambda=\lambda$ and $\mathrm{ran} \, \mu \subseteq \mathrm{ran} \, \lambda$, then $(\W,\W,\mu)$ is said to be a {\it reduced system} of $(\V,\W,\lambda)$.
\end{itemize}

\gap

In a follow up paper, we plan to consider transfer principles. A typical result here says that if $(\W,\W,\mu)$ is a reduced system of $(\V,\W,\lambda)$ and a set $Q$ in $\W$ has some specific property (e.g., closed, open, convex, and/or compact) then $\lambda^{-1}(Q)$ will also have the same property in $\V$. We also plan to  study subdifferentials of certain functions and  extend a result of Lewis from normal decomposition systems  to the setting of FTvN systems.

\gap

The organization of the paper is as follows. 
In Section 2, we cover preliminary definitions regarding FTvN systems; numerous examples are given in Section 3. Section 4 covers some basic properties. By defining a spectral set in $\V$ as a set of the form $\lambda^{-1}(Q)$ for some $Q\subseteq \W$, we state  a result that shows the invariance of the spectral property  under certain topological and algebraic operations. In Section 5, we introduce the concept of center of a FTvN system  and establish several results. We show that  the center is a closed subspace of $\V$ and prove a  result wherein the given FTvN system is decomposed as the orthogonal `sum' of a FTvN system with `trivial' center and another with `full' center. We also introduce the concept of a unit element and show that in the settings of Euclidean Jordan algebras and hyperbolic systems, certain familiar objects are unit elements. Section 6 covers automorphisms of FTvN systems. We show that in the setting of Euclidean Jordan algebras, algebra automorphisms coincide with automorphims of the corresponding FTvN system. We introduce the concept of orbit-transitive FTvN system and describe a result characterizing a FTvN system as a normal decomposition system. In Sections 7 and 8, we cover the concepts of  majorization and doubly stochastic transformations. Section 9 deals with reduced systems. Finally, in the Appendix, we recall the definitions of normal decomposition systems, Eaton triples, and the rearrangement inequality for measurable functions.

%%%%%%%%%%%%%%%%%%%%%%%%%%%%%%%%%%
\section{Preliminaries}
Throughout this paper, we deal with real inner product spaces. In any given real inner product space, we write $\ip{x}{y}$ for  the inner product between two elements $x$ and $y$ and $\norm{x}$ for the (corresponding) norm of $x$. Also,  for a set $S$ in such a space, we write $\overline{S}$, $S^\circ$, $\partial(S)$, and $S^\perp$ for the {\it closure, interior, boundary, and orthogonal complement} of $S$, respectively. We also write $\conv(S)$ (or $\conv\,S$) and $\spn(S)$ for the {\it convex hull} and the {\it span} of $S$, respectively. We say that a nonempty set $S$  is a {\it cone} if $tx\in S$ for all $x\in S$ and $t\geq 0$ in $\R$. We write $\cone(S):=\{tx:t\geq 0,\,x\in S\}$ for the cone generated by $S$.

\gap

Throughout the paper, we will be using  some standard convexity results. For ease of reference, we state them here.

\begin{theorem} \label{minkowski} (\cite{rudin}, Theorem 3.25) \\
	In a finite dimensional Hilbert space, the convex hull of a compact set is compact.
\end{theorem}

\begin{theorem} (Supporting hyperplane theorem, \cite{rockafellar}, p.100) \label{supporting hyperplane theorem} \\
	Suppose $\V$ is a finite dimensional Hilbert space, $K$ is a closed convex set in $\V$, and $u\in \partial(K)$. Then, there exists a nonzero $c\in \V$ such that
	\[ \ip{c}{u} \geq \ip{c}{x} \,\,\text{for all}\,\, x \in K. \]
\end{theorem}

The following is a specialized version of the well-known Hahn-Banach theorem (\cite{rudin}, Theorem 3.4) coupled with the Riesz representation theorem.

\begin{theorem} (Strong separation theorem) \label{strong separation theorem} \\
	Suppose $\V$ is a Hilbert space, $u\in \V$, and 
	$K$ is a closed convex set in $\V$. If $u\not\in K$, then there exist  $c\in \V$  and $\alpha\in \R$ such that
	\[ \ip{c}{u} >\alpha \geq \ip{c}{x} \,\, \text{for all} \,\, x \in K. \]
	In particular, if $K$ is also a cone, we can take $\alpha=0$. 
\end{theorem}

%%%%%%%%%%%%%%%%%%%%%%%%%%%%%%%%%%%%%%%%%%%%%%%%%%%%%%%%%%%%%%%%%
\section{FTvN systems: Basic definitions}

We start with an expanded version of the definition of a FTvN system.

\begin{definition} \cite{gowda-ftvn} \label{ftvn}
	A Fan-Theobald-von Neumann system (FTvN system, for short) is a triple $(\V,\W,\lambda)$, where $\V$ and $\W$ are real inner product spaces and $\lambda:\V \to \W$ is a map satisfying the following conditions:
	\begin{itemize}
		\item [$(A1)$] $\norm{\lx} = \norm{x}$ for all $x\in \V$.
		\item [$(A2)$] $\ip{x}{y} \leq \ip{\lx}{\ly}$ for all $x,y\in \V$.
		\item [$(A3)$] For any given $c\in \V$ and $q\in \lambda(\V)$, there exists $x\in \V$ such that
		\begin{equation}\label{A3}
			\lx=q \,\, \text{and} \,\, \ip{c}{x} = \ip{\lc}{\lx}.
		\end{equation}
	\end{itemize}
\end{definition}

Here, $(A1)$ refers to the norm-preserving property of $\lambda$;  properties $(A2)$ and $(A3)$ together describe (\ref{intro ftvn}). We note that $(A1)$ is, actually, a consequence of $(A2)$ and $(A3)$. This is seen as follows.
Consider any $u\in \V$. Letting $x=y=u$ in $(A2)$, we get $||u||\leq ||\lambda(u)||$. On the other hand, letting $c=u$ and $q=\lambda(u)$ in $(A3)$, we get an $x\in \V$ such that $\lambda(x)=\lambda(u)$ and $\ip{u}{x}=\ip{\lambda(u)}{\lx}$. Then,
\[ \ip{\lambda(u)}{\lambda(u)}=\ip{\lambda(u)}{\lx}=\ip{u}{x}\leq ||u||\,||x||\leq ||u||\,||\lambda(x)||=||u||\,||\lambda(u)||, \]
where the first inequality is due to the Cauchy-Schwarz inequality. This gives $||\lambda(u)||\leq ||u||$. As $u\in \V$ is arbitrary, we have $||\lambda(u)||=||u||$ for all $u \in \V$.

\gap

Let $(\V,\W,\lambda)$ be a FTvN system. We denote the {\it range} of $\lambda$ by $\mathrm{ran}\,\lambda$; the {\it $\lambda$-orbit} of an element $u\in \V$ is defined by 
\[ [u]:=\{x\in \V:\lx=\lu\}. \] 
More generally, for any set $S\subseteq \V$,
\[ [S]:=\underset{u\in S}{\bigcup}\,[u]. \]
We write $x\sim y$ if $\lx=\ly$, or equivalently, $[x]=[y]$. This defines an equivalence relation on $\V$.

\begin{itemize}
	\item A set $E$ in $\V$ is said to be a {\it spectral set} if it is of the form 
	$E=\lambda^{-1}(Q)$ for some $Q\subseteq \W$, or equivalently, a union of $\lambda$-orbits. It is easy to see that a set $E$ is a spectral set if and only if the implication $x\in E \Rightarrow [x]\subseteq E$ holds.
	
	\item A spectral set that is also a convex cone (that is, it is closed under nonnegative linear combinations) is said to be a {\it spectral cone}.
	
	\item A real-valued function $\Phi:\V \to \R$ is a {\it spectral function} if it is of the form $\Phi=\phi\circ \lambda$ for some function $\phi:\lambda(\V) \to \R$, or equivalently, $\Phi$ is a constant function on every $\lambda$-orbit.
\end{itemize}

\begin{definition}
    We say that elements $x,y\in \V$ {\it commute} in the FTvN system $(\V,\W,\lambda)$ if
    \[ \ip{x}{y} = \ip{\lx}{\ly}. \]
\end{definition}

\gap

%%%%%%%%%%%%%%%%%%%%%%%%%%%%%%%%%%%%%%%%%%%%%%%%%
\section{FTvN systems: Examples}
In what  follows, we present some examples of FTvN systems; a few come  from \cite{gowda-ftvn}. 

\begin{example} \label{ips} ({\it Real inner product space}\,) 
	Consider a real inner product space $(\V, \ip{\cdot}{\cdot})$ with the corresponding norm $\norm{\cdot}$. Then, with $\lx := \norm{x}$, the triple $(\V,\R,\lambda)$ is a FTvN system. In this system, any $\lambda$-orbit is either $\{0\}$ or a sphere centered at the origin. A real-valued function is spectral if it is radial (that is, it depends only on the norm). Two elements commute if and only if one is a nonnegative multiple of the other.
\end{example}

\begin{example}\label{discrete space} ({\it Discrete system}\,)
	On a real inner product space $\V$, let $S$ be a linear isometry, that is, $S$ is linear and $\norm{Sx} = \norm{x}$ for all $x \in \V$. Then, with $\lx:=Sx$, $(\V,\V,\lambda)$ becomes a FTvN system. 
	Here, as $S$ is one-to-one, every $\lambda$-orbit is a singleton; hence, every set in $\V$ is a spectral set. Moreover, any two elements commute (as $S$ preserves inner products). We will call this system a {\it discrete system.}
\end{example}

\begin{example} \label{rn} ({\it The Euclidean space $\Rn$}\,)
	Let $\V=\Rn$, equipped with the usual inner product. Corresponding to a vector $x = (x_1, x_2, \ldots, x_n)$ in $\Rn$, let $x^\downarrow$ denote the decreasing  rearrangement of $x$, that is, $x^\downarrow:=(x_1^\downarrow, x_2^\downarrow,\ldots, x_n^\downarrow)$ with $x_1^\downarrow\geq x_2^\downarrow\geq \cdots \geq x_n^\downarrow$. We also let $|x|$ denote the vector of absolute values of entries of $x$. It is easy to see that $|x|^\downarrow=|x^\downarrow|^\downarrow$.
	For any $x \in \Rn$, let $\lx := x^\downarrow$ and $\mu(x) := |x|^\downarrow$. Then, $(\Rn,\Rn,\lambda)$ is a FTvN system with the FTvN inequality reducing to the {\it Hardy-Littlewood-P\'{o}lya rearrangement inequality}
	\[ \ip{x}{y} \leq \big\langle x^\downarrow, y^\downarrow \big\rangle \quad (x, y \in \Rn). \]
	In this system, two elements $x$ and $y$ commute if and only if there is a permutation that simultaneously takes $x$ to $x^\downarrow$ and $y$ to $y^\downarrow$.  The triple $(\Rn,\Rn,\mu)$ also a FTvN system. These two are particular instances of (Euclidean Jordan algebra) examples given below.
\end{example}

\begin{example}\label{sn,hn} ({\it The set of $n \times n$ real symmetric/complex Hermitian matrices}\,)
	Let $\Sn$ ($\Hn$) denote the space of all $n\times n$ real symmetric (respectively, complex Hermitian) matrices. These spaces carry the real inner product $\ip{X}{Y} := \tr(XY)$, where `tr' denotes the trace. For $X$ in $\Hn$, let $\lambda(X)$ denote the vector of eigenvalues of $X$ written in the decreasing order. From the spectral decomposition theorem, we see  that $||X||=||\lambda(X)||$ for all $X$; so, $\lambda$ map is norm preserving.
	
	In \cite{fan}, Theorem 1, Fan proved that for any $n\times n$ Hermitian matrix $X$ and $1\leq k\leq n$,
	the sum $\sum_{i=1}^{k}\lambda_i(X)$ is the maximum of $\sum_{i=1}^{k} \ip{Xu_i}{u_i}$ with  $k$ orthonormal vectors $u_1, u_2, \ldots, u_k$ varying in the underlying space. We can interpret this result as:
	\[ \underset{U}{\max} \ip{X}{UI_kU^*} = \ip{\lambda(X)}{\lambda(I_k)}, \]
	where $I_k$ is the diagonal matrix with $k$ leading $1$s and $0$s elsewhere, and $U$ denotes a unitary/orthogonal matrix. Noting that the set of matrices $UI_kU^*$ is precisely the $\lambda$-orbit $[I_k]$, the above statement verifies $(1)$ in one special case.
	 
	This result proves Schur's theorem on majorization (that the diagonal of a Hermitian matrix is majorized by the eigenvalue vector) as well as the inequality
	\begin{equation}\label{fan-particular case}
		\ip{X}{I_k} \leq \ip{\lambda(X)}{\lambda(I_k)}.
	\end{equation}
	As an extension of (\ref{fan-particular case}), {\it Richter} \cite{richter} proves the following:
	\begin{equation}\label{richter}
		\ip{X}{Y} \leq \ip{\lambda(X)}{\lambda(Y)} \quad (X, Y \in \Hn).
	\end{equation}
	Noting that Richter's proof is somewhat analytical, Mirsky \cite{mirsky} gives a simple  algebraic proof based on the above result of Fan. 
	
	Later, {\it Theobald} \cite{theobald} showed that equality holds in (\ref{richter}) if and only if there exists a unitary/orthogonal matrix $U$ such that
	\[ X = U\mathrm{diag}(\lambda(X))U^* \quad \text{and} \quad Y = U \mathrm{diag}(\lambda(Y))U^*. \]
	Putting all these together, we see that $(\Sn, \Rn, \lambda)$ and $(\Hn, \Rn, \lambda)$ are FTvN systems.
\end{example}

\begin{example} \label{space mn} ({\it The space $M_n$}\,)
	Consider the triple $(M_n,M_n,\lambda)$, where $M_n$ denotes the set of all $n\times n$ complex matrices with real inner product
	$\ip{X}{Y} := Re \tr(X^*Y)$, and the map $\lambda$ takes $X$ in $M_n$ to $\diag (s(X))$ (the diagonal matrix consisting of the singular values of $X$ written in the decreasing order). 
	
	In this setting, {\it von Neumann} \cite{neumann} proved the following: For all $X,Y\in M_n$,
	\[ \ip{X}{Y} \leq \ip{\lambda(X)}{\lambda(Y)}, \]
	with equality if and only if there exist unitary matrices $U$ and $V$ such that 
	\[ X = U\mathrm{diag}(s(X))V \quad \text{and} \quad Y = U\mathrm{diag}(s(Y))V. \]
	This result shows that $(M_n, M_n, \lambda)$ is a FTvN system.
\end{example}

\begin{example} \label{eja} ({\it Euclidean Jordan algebra}\,) 
	Let $\V$ be a Euclidean Jordan algebra of rank $n$ carrying the trace inner product and $\lambda : \V \to \Rn$ denotes the eigenvalue map (so, for any $x \in \V$, $\lx$ denote the vector of eigenvalues of $x$ written in the decreasing order). Then, $\lambda$ is a norm-preserving map and the inequality 
	\[ \ip{x}{y} \leq \ip{\lx}{\ly} \quad (x, y\in \V) \]
	is known \cite{lim et al, baes, gowda-tao}. Additionally, for any $c\in \V$ and $q=\lu$ for some $u\in \V$, we can write $c = \lambda_1(c)e_1 + \lambda_2(c)e_2 + \cdots + \lambda_n(c)e_n$, where  $\{e_1, e_2, \ldots, e_n\}$ is a Jordan frame, and define
	\[ x =\lambda_1(u)e_1 + \lambda_2(u)e_2 + \cdots + \lambda_n(u)e_n. \]
	Then, $\lx=\lu=q$ and, due to the orthonormality of any Jordan frame,   $\ip{c}{x} = \ip{\lc}{\lx}$. These arguments show that conditions in Definition \ref{ftvn} are verified. Thus, $(\V,\Rn,\lambda)$ is a FTvN system.
	In this setting, a set in $\V$ is a spectral set if it is of the form $\lambda^{-1}(Q)$ for some (permutation invariant) set $Q$ in $\Rn$; a function $\Phi:\V \to \R$ is a spectral function if it is of the form $\phi \circ \lambda$ for some (permutation invariant) function $\phi:\Rn \to \R$. 
	When $\V$ is a simple Euclidean Jordan algebra or equivalent to $\Rn$, these are precisely sets and functions that are invariant under algebra automorphisms of $\V$ \cite{jeong-gowda-spectral set}. (Algebra automorphisms of $\V$ are invertible linear transformations that preserve the Jordan product.) Commutativity of elements $x$ and $y$ in the FTvN system $(\V,\W,\lambda)$ means that there is a Jordan frame $\{e_1,e_2,\ldots, e_n\}$ in $\V$ such that $x = \lambda_1(x)e_1 + \lambda_2(x)e_2 + \cdots + \lambda_n(x)e_n$ and 	$y = \lambda_1(y)e_1 + \lambda_2(y)e_2 + \cdots + \lambda_n(y)e_n$. This defines the {\it strong operator commutativity} in the algebra $\V$. The algebras of $n \times n$ real/complex Hermitian matrices are primary examples of Euclidean Jordan algebras of rank $n$.
\end{example}

\begin{example} \label{eja with absolute value} ({\it Euclidean Jordan algebra with absolute value map}\,) 
	Let $\V$ be a Euclidean Jordan algebra of rank $n$ carrying the trace inner product and $\lambda:\V \to \Rn$ denote the eigenvalue map. Then, as noted in Example \ref{eja}, $(\V, \Rn, \lambda)$ is a FTvN system. 
	Now define $\mu : \V \to \Rn$ by $\mu(x) := |\lx|^\downarrow$. {\it We claim that $(\V,\Rn, \mu)$ is also a FTvN system.} Clearly, conditions $(A1)$ and $(A2)$ in Definition \ref{ftvn} hold. We verify $(A3)$. Let $c\in \V$ and $q=\mu(u)=|\lu|^\downarrow$ for some $u\in \V$. Let the spectral decomposition of $c$ in $\V$ be given by $c = c_1e_1 + c_2e_2 + \cdots + c_ne_n$, where $c_1, c_2, \ldots, c_n$ denote the eigenvalues of $c$ and $\{e_1, e_2, \ldots, e_n\}$ is a Jordan frame. We define $\varepsilon_i: = \mathrm{sign} \, c_i$ for $i = 1, 2, \ldots, n$ so that 
	\[ c = \varepsilon_1|c_1|e_1 + \varepsilon_2|c_2|e_2 + \cdots + \varepsilon_n|c_n|e_n. \]
	By rearranging $e_k$s, if necessary, we may assume that $|c_1| \geq |c_2| \cdots \geq |c_n|$, in which case, $\mu(c)=(|c_1|, |c_2|, \ldots, |c_n|)$. Now, corresponding to $q=|\lu|^\downarrow$, we define
	\[ x = \varepsilon_1q_1e_1 + \varepsilon_2q_2e_2 + \cdots + \varepsilon_n q_ne_n. \]
	Then, $\mu(x) = |\lx|^\downarrow=q$ (observe that entries of $q$ are nonnegative and decreasing). Furthermore, due to the orthonormality of the Jordan frame $\{e_1,e_2,\ldots, e_n\}$, we see that $\ip{c}{x} = \sum_{i=1}^{n}|c_i|q_i = \ip{\mu(c)}{\mu(x)}$. This verifies $(A3)$. We conclude that $(\V,\Rn,\mu)$ is a FTvN system. 
	It is easy to verify that if $x$ and $y$ commute in $(\V,\Rn,\mu)$, then they do so in $(\V,\Rn,\lambda)$.
\end{example}

\begin{example} \label{hyperbolic} ({\it System induced by a hyperbolic polynomial} \cite{bauschke et al})
	Let $\V$ be a finite dimensional real vector space, $e\in \V$, and $p$ be a real homogeneous polynomial of degree $n$ on $\V$. We say that $p$ is {\it hyperbolic} with respect to $e$ if $p(e)\neq 0$ and for every $x\in \V$,  the roots of the univariate polynomial $t \mapsto p(te-x)$ are all real. Let $p$ be such a polynomial. For any $x\in \V$, let $\lx$ denote the vector of roots of this univariate polynomial with entries written in the decreasing order. When $p$ is {\it complete} (which means that $\lx=0$ implies $x=0$), $\V$ becomes an inner product space under the inner product
	\[ \ip{x}{y} := \frac{1}{4} \Big[ \norm{\lambda(x+y)}^2 - \norm{\lambda(x-y)}^2 \Big], \]
	(where the right-hand side is computed in $\Rn$ with the usual norm), see \cite{bauschke et al}, Theorem 4.2. Relative to this inner product, $\lambda : \V \to \Rn$ becomes norm-preserving and $\ip{x}{y} \leq \ip{\lx}{\ly}$ for all $x,y\in \V$, see \cite{bauschke et al}, Proposition 4.4. 
	When $p$ is also {\it isometric} (which means that for every $c\in \V$ and $q \in \lambda(\V)$, there exists $x\in \V$ such that $\lx=q$ and $\lambda(c+x)=\lc+\lx$). In this case, it is shown in \cite{bauschke et al}, Proposition 5.3, that (\ref{intro ftvn}) holds. Thus, when the hyperbolic polynomial $p$ is complete and isometric, the triple $(\V,\Rn,\lambda)$ becomes a FTvN system.
	
	We note that the previous example (Euclidean Jordan algebra, Example \ref{eja}) becomes a special case via the polynomial
	$p(x):=\det x$ (the product of eigenvalues of $x$). While a number of other examples are known \cite{bauschke et al}, the authors are not aware of any  characterization of the isometric property. 
\end{example}

\begin{example} \label{nds} ({\it Normal decomposition system}\,)
	Consider a normal decomposition system $(\V,\G,\gamma)$ \cite{lewis}. Here $\V$ is a real inner product space, $\G$ is a closed subgroup of the orthogonal group of $\V$, and the map $\gamma:\V \to \V$ satisfies certain conditions, see the Appendix. 
	Then, with $\W = \spn(\gamma(\V))$ (which is a subset of $\V$) and $\lambda = \gamma$, the triple $(\V,\W,\lambda)$ becomes a FTvN system \cite{gowda-ftvn}. In this system, $\lambda^2=\lambda$ and spectral sets/functions are those that are invariant under elements of $\G$. Also, $x$ and $y$ commute in $(\V,\W,\lambda)$ if and only if there exists $A\in \G$ such that $x=A\gamma(x)$ and $y=A\gamma(y)$. The space of all $n\times n$ complex matrices is a primary example of a normal decomposition system, see Example \ref{space mn}.
	It is known, see \cite{lim et al}, that every simple Euclidean Jordan algebra can be regarded as a normal decomposition system.
\end{example}

\begin{example} ({\it Eaton triple}\,) 
	Consider an Eaton triple $(\V,\G,F)$ \cite{eaton-perlman} (see the Appendix for the definition). Here, $\V$ is a finite dimensional real inner product space, $\G$ is a closed subgroup of the orthogonal group of $\V$, and $F$ is a closed convex cone in $\V$. For any $x\in \V$, let $\mathrm{Orb}(x) = \{Ax:A\in \G\}$. Then with $\W:=F-F$ and $\lx$ denoting the unique element in the singleton set $\mathrm{Orb}(x)\cap F$, the triple $(\V,\W,\lambda)$ becomes a FTvN system \cite{gowda-ftvn}. 
	It is known that every Eaton triple is a normal decomposition system and every finite dimensional normal decomposition system is an Eaton triple.
\end{example}

\begin{example} \label{sequence space} ({\it The sequence space $\ell_2(\R)$}\,)
	Let $\V$ denote the space $\ell_2(\R)$ of  all of square summable real sequences  with the usual inner product and norm. We show that $\V$ can be made into a FTvN system. First, let ${\cal F}$ denote the subspace of all ``finite" sequences:  $x=(x_k)\in {\cal F}$ if either $x$ is zero or there exists a natural number $n \in \mathbb{N}$ such that $x_k=0$ for all $k\geq n$.  If such an $x$ is nonzero, we write $x=(\overline{x},\overline{0})$, where $\overline{x}\in \Rn$  and $\overline{0}$ is a sequence of zeros. (Note that this representation is not unique.) We then consider, as in Example \ref{rn}, $|\bar{x}|^\downarrow$ in $\Rn$ and 
	\[ \lx:=|x|^\downarrow:=\big( |\overline{x}|^\downarrow, \overline{0} \big), \]
	which is uniquely defined as an element of  $\ell_2(\R)$. We also let $\lambda(0)=0$. Now, given $x,y\in {\cal F}$, we write $x = (\overline{x}, \overline{0})$ and $y = (\overline{y}, \overline{0})$ where $\overline{x}, \overline{y}\in \Rn$ for some $n$. Then,
	\[ \ip{x}{y} = \ip{\overline{x}}{\overline{y}} \leq \big\langle |\overline{x}|^\downarrow, |\overline{y}|^\downarrow \big\rangle = \ip{\lx}{\ly}. \]
	(Note that inner products are computed in different spaces; for the inequality, see Example \ref{rn}.) Moreover, for all $x,y\in {\cal F}$,
	\[ ||\lx-\ly||=|| \,|\overline{x}|^\downarrow-|\overline{y}|^\downarrow\,||\leq ||\overline{x}-\overline{y}||=||x-y||.\]
	The above Lipschitzian property implies that $\lambda$ is uniformly continuous on ${\cal F}$. Since any element $x=(x_1,x_2,\ldots)\in \ell_2(\R)$ can be written as the limit of $x^{(n)}:=(x_1,x_2,\ldots, x_n,0,0,\ldots)\in {\cal F}$, We can extend $\lambda$ to $\ell_2(\R)$ uniquely; this extension -- still denoted by $\lambda$ -- satisfies the following properties:
	\begin{itemize}
		\item [$(i)$] $\lambda$ is norm-preserving.
		\item [$(ii)$] $\ip{x}{y} \leq \ip{\lx}{\ly}$ for all $x,y\in \ell_2(\R)$.
		\item [$(iii)$] The components of  $\lx$ are nonnegative and decreasing.
		\item [$(iv)$] If $x=(x_n)\in {\cal F}$, then $\lx\in {\cal F}$.
		\item [$(v)$] If $x=(x_n)$ has infinitely many nonzero entries, say, $x_{n_1},x_{n_2},\ldots$, then there is a permutation $\sigma$ on $\{n_1,n_2,\ldots\}$ such that 
		$\lx=\big (|x_{\sigma(n_1)}|, |x_{\sigma(n_2)}|,\ldots\big )$. In particular, in this case, all entries of $\lx$ are positive. 
	\end{itemize}
	
	We remark that the permutation $\sigma$ in Item $(v)$ can be constructed by observing that each interval
	of the form $\left( \frac{||x||}{2^k},\, \frac{||x||}{2^{k-1}} \right]$ contains a finite number of nonzero entries of $x$. Another (formal and known) way of constructing $\lx$ for any $x\in \ell_2(\R)$ is by letting:
	$\lx=x^*$, where 
	\[ x_n^* = \inf \{ \alpha \geq 0 : \mu \big( \{k \in \mathbb{N} : \abs{x_k} > \alpha \} \big) \leq n-1 \} \;\; (\forall\, n \in \mathbb{N}),\]
	with $\mu$ denoting the counting measure on $\mathbb{N}$, see the Appendix.
	
	Now that conditions $(A1)$ and $(A2)$ of Definition \ref{ftvn} are verified, see Items $(i)$ and $(ii)$ above, we  show that condition $(A3)$ also holds.  
	Suppose we are given $c := (c_n) \in \V$ and $q := (q_n) \in \lambda(\V)$; note that $q_{n} \geq q_{n+1} \geq 0$ for all $n \in \mathbb{N}$. Define the set 
	$D = \{ n \in \mathbb{N} : c_n = 0\}$.
	
	\gap
	
	{\it Case 1:}  $\mathbb{N} \setminus D$ is finite, i.e., there are only finitely many, say, $k$ nonzero entries in $c$. Let $c_{n_1},c_{n_2},\ldots, c_{n_k}$ be these entries.
	We rearrange  these entries of $c$ (for example, by relabeling the indices $n_1,n_2,\ldots, n_k$) so that  $|c_{n_1}| \geq \cdots \geq |c_{n_k}|$, in which case,
	$\lc=\big (|c_{n_1}|, |c_{n_2}|,\ldots, |c_{n_k}|,0,0,\ldots \big )$.  Let $\varepsilon_{n_i}$ denote the sign of $c_{n_i}$. 
	Define $x := (x_n)$ in $\V$ such that $x_{n_i} = \varepsilon_{n_i}\,q_i$ for $i = 1, \ldots, k$ with  all the other entries of $x$ taken (in order) from $q_{k+1}, q_{k+2}, \ldots$. Then, $\lx = q$ and
	\[ \ip{c}{x} = \sum_{n=1}^{\infty} c_n x_n = \sum_{i=1}^{k} c_{n_i} x_{n_i} = \sum_{i=1}^{k} \varepsilon_{n_i} |c_{n_i}| \varepsilon_{n_i} q_i = \sum_{i=1}^{k} |c_{n_i}| q_i = \ip{\lc}{q}. \]
	
	\gap
	
	{\it Case 2:}  $\mathbb{N} \setminus D$ is countably infinite.  Then there exists a bijection $\delta : \mathbb{N} \to \mathbb{N} \setminus D$. Note that $\tilde{c} := (\tilde{c}_n)$, defined by $\tilde{c}_n = c_{\delta(n)}$, is a sequence without zero entries and $\norm{\tilde{c}} = \norm{c}$.
	
	Now, for each $n \in \mathbb{N}$, write $\tilde{c}_n = \varepsilon_n \abs{\tilde{c}_n}$, where $\varepsilon_n$ is the sign of $\tilde{c}_n$. Consider $\tilde{c}^* := \lambda(\tilde{c})$; by Item $(v)$ above,  there exists a bijection $\sigma : \mathbb{N} \to \mathbb{N}$ such that $\tilde{c}_n^* = \abs{\tilde{c}_{\sigma(n)}}$ for all $n \in \mathbb{N}$. Define $x := (x_n)$ in $\V$ by
	\[ x_n = \begin{cases}
		\varepsilon_{\delta^{-1}(n)} q_{\sigma^{-1}(\delta^{-1}(n))} & \text{if} \;\; n \notin D, \\
		0 & \text{if} \;\; n \in D.
	\end{cases} \]
	
	Note that $x_{\delta(n)} = \varepsilon_n q_{\sigma^{-1}(n)}$ for all $n \in \mathbb{N}$ and $\lx = q$, as $\mathbb{N} \setminus D$ is countably infinite. Thus, $\{\abs{x_n} : n \in \mathbb{N} \setminus D\}$ contains all the entries of $q$. Moreover, we have
	\begin{align*}
		\ip{c}{x} 
		& = \sum_{n=1}^{\infty} c_n x_n                              
		= \sum_{n=1}^{\infty} c_{\delta(n)} x_{\delta(n)}
		= \sum_{n=1}^{\infty} \varepsilon_n \abs{\tilde{c}_n} \varepsilon_n  q_{\sigma^{-1}(n)} \\
		& = \sum_{n=1}^{\infty} \abs{\tilde{c}_n} q_{\sigma^{-1}(n)} 
		= \sum_{n=1}^{\infty} \abs{\tilde{c}_{\sigma(n)}} q_{n}
		= \sum_{n=1}^{\infty} \tilde{c}_n^* q_n = \ip{\lc}{q}.
	\end{align*}
	Thus we have verified condition $(A3)$. Hence, $(\V, \V, \lambda)$ is a FTvN system.
\end{example}

{\bf Remark.} From the {\it Case 1} above, we see that $({\cal F},{\cal F},\lambda)$ is also a FTvN system. If $H$ is an infinite dimensional real separable Hilbert space, then $H$, being isometrically isomorphic to $\ell_2(\R)$, can be made into a FTvN system.

\gap

{\bf Problem.} Consider the Hilbert space $\V$ of real square integrable functions over a $\sigma$-finite measure space. Corresponding to $f\in \V$, define the decreasing rearrangement $f^*$, see the Appendix. With $\lambda(f)=f^*$, the problem is to decide if (or  when) $(\V,\V,\lambda)$ is a FTvN system.

\begin{example} \label{twisted FTvn system} ({\it Twisted FTvN system}\,)
	Consider a FTvN system $(\V,\W,\lambda)$. Define $\widetilde{\lambda}:\V\to \W$ by 
	\[ \widetilde{\lambda}(x):=-\lambda(-x)\quad (x\in \V). \]
	It is easy to see that $(\V,\W,\widetilde{\lambda})$ is also a FTvN system. (In Example \ref{rn}, if $\lambda$ denotes the decreasing rearrangement, then $\widetilde{\lambda}$ denotes the increasing rearrangement.)
	We note that $E$ is spectral in $(\V,\W,\lambda)$ if and only if $-E$ is spectral in $(\V,\W,\widetilde{\lambda})$. While in certain settings (Examples \ref{eja} and \ref{hyperbolic}) $E$ is spectral if and only if $-E$ is also spectral in the same system, it is not clear if this holds in general FTvN systems. 
\end{example}

\begin{example} \label{cartesian prouct} ({\it Cartesian product}\,)
	A Cartesian product of (a finite number of) FTvN systems can be made into a FTvN system in an obvious way. 
	For example, if $(\V_1,\W_1,\lambda_1)$ and $(\V_2,\W_2,\lambda_2)$ are two FTvN systems, then so is $(\V_1\times \V_2,\W_1\times \W_2,\lambda_1\times \lambda_2)$, where
	the inner product in $V_1\times \V_2$ is defined by
	\[ \big\langle (x_1, x_2), (y_1, y_2) \big\rangle := \ip{x_1}{y_1} + \ip{x_2}{y_2} \]
	(with a similar definition in $\W_1\times \W_2$), and 
	\[ (\lambda_1\times \lambda_2)(v_1,v_2):=\,\big (\lambda_1(v_1),\lambda_2(v_2)\,\big ). \] 
	We consider two particular instances:
	\begin{itemize}
		\item [$\bullet$] Given any FTvN system $(\V, \W, \lambda)$, let $(\V_1, \W_1, \lambda_1)= (\R, \R, \tau)$ with $\tau(t)=t$ for all $t\in \R$ and $(\V_2,\W_2,\lambda_2)=(\V,\W,\lambda)$. We can then form the product FTvN system $(\R\times \V,\R\times \W,\tau\times \lambda)$, where
		\[ (\tau \times \lambda)(t,v) = (t, \lv) \quad \big( (t,v) \in \R \times \V \big). \]
		
		\item [$\bullet$] Consider a FTvN system $(\V,\W,\lambda)$ and take  $w_0\in \W$ with $||w_0||=1$. Then $(\R, \W, \nu)$ is a FTvN system, where $\nu(t)=tw_0$ for all $t\in \R$. We can then form the system $(\R\times \V,\W\times \W,\nu\times \lambda)$, where
		\[ (\nu \times \lambda)(t,v) = (tw_0, \lv) \quad \big( (t,v) \in \R \times \V \big). \]
	\end{itemize}
\end{example}

\begin{example} \label{composition} ({\it Composition}\,)
	Suppose $(\V,\W,\lambda)$ is a FTvN system and $S:\W \to \W$ is a linear isometry. It  is easy to see that $(\V,\W,S\circ \lambda)$ is also a FTvN system. 
	Now, starting with a FTvN system $(\V,\W,\lambda)$ and $w_0\in \W$ with $||w_0||=1$, we can form the 
	product system $(\R\times\V,\W\times \W,\nu\times \lambda)$, see the previous example. Now consider the linear isometry $S$ on $\W\times \W$ defined by 
	\[ S : (w_1, w_2) \mapsto \frac{1}{\sqrt{2}} \Big( w_1 + w_2, w_1 - w_2 \Big). \]
	Then, writing $\mu=S\circ (\nu\times \lambda)$, we see that
	\[ \mu(t,v) := \frac{1}{\sqrt{2}} \Big( tw_0 + \lv, tw_0 - \lv \Big). \] 
	Thus, we get the new FTvN system $(\R\times \V,\W\times \W,\mu)$.
	
	Specializing further, if $\V$ is an inner product space, then with $\W=\R$, $w_0=1$, and $\lx = \norm{x}$, we get the FTvN system $(\R\times \V,\R^2,\mu)$, where
	\[ \mu(t,v) := \frac{1}{\sqrt{2}} \Big( t+\norm{v}, t-\norm{v} \Big). \]
	One may note the similarity between the above $\mu$ and the eigenvalue map that appears in the Jordan spin algebra ${\cal L}^n$ in the study of Euclidean Jordan algebras.
\end{example}

\begin{example} \label{subspace of a FTvN} ({\it Subspace of a FTvN system}\,) 
	Let $(\V,\W,\lambda)$ be a FTvN system. If $\mathcal{U}$ is a (linear) subspace of $\V$ that is  spectral, then $(\mathcal{U},\W,\lambda)$ is  a FTvN system. However, this conclusion may not hold if $\mathcal{U}$ is not spectral. We provide an example taken from \cite{bauschke et al}, appropriately modified. Consider $(\R ^3,\R ^3,\lambda)$ with $\lx=x^\downarrow$, see Example \ref{rn}. Let $\mathcal{U}$ be the span of vectors  $(1, 1, 1)$ and $(3, 1,0)$. Then, for any $x=\alpha (1, 1, 1) + \beta (3, 1, 0)$ in $\mathcal{U}$, with $\alpha,\beta\in \R$, $\lx=x$ if $\beta \geq 0$ and $\lx= \alpha (1, 1, 1) + \beta (0, 1, 3)$ if $\beta<0$.
	Now, $\lambda : \mathcal{U} \to \R^3$ satisfies conditions $(A1)$ and $(A2)$ in Definition \ref{ftvn}. However, on $\mathcal{U}$, $\lx=\ly \Rightarrow x=y$  and condition $(A3)$ fails to hold with  $c=(3,1,0)$, $u=-(3,1,0)$, and $q=\lu=(0,-1,-3)$.
\end{example}

We end this section with a problem.

\gap

{\bf Problem.} Suppose $(\V, \W, \lambda)$ is an FTvN system. Let $\overline{\V}$ and $\overline{\W}$ be the completions of (the inner product spaces) $\V$ and $\W$, respectively. As $\lambda:\V\to \W$ is Lipschitz, see Theorem \ref{basic theorem}, it can be extended to $\overline{\lambda} : \overline{\V} \to \overline{\W}$. It is easy to see that conditions $(A1)$ and $(A2)$ of Definition \ref{ftvn} hold for $\overline{\lambda}$. The problem is to decide whether condition $(A3)$ also holds. 
See Example \ref{sequence space} for motivation.

%%%%%%%%%%%%%%%%%%%%%%%%%%%%%%%%%%%%%%%%%%%%%%%%
\section{FTvN systems: some basic properties}

In this section, we describe some basic properties that hold in FTvN systems.

\begin{theorem} (\cite{gowda-ftvn}, Section 2) \label{basic theorem}
	Let $(\V,\W,\lambda)$ be a FTvN system. Then, the following hold for $x,y,c\in \V$:
	\begin{itemize}
		\item [$(a)$] $\lambda(tx)=t\lx$ for all $t\geq 0$.
		\item [$(b)$] $\norm{\lx - \ly} \leq \norm{x-y}$.
		\item [$(c)$] $\ip{\lc}{\lambda(x+y)} \leq \ip{\lc}{\lx} + \ip{\lc}{\ly}$. More generally, 
		for $c,x_1,x_2,\ldots, x_k$ in $\V$,
		\begin{equation} \label{general sublinearity}
			\big\langle \lc, \lambda(x_1 + x_2 + \cdots + x_k) \big\rangle \leq \big\langle \lc, \lambda(x_1) + \lambda(x_2) + \cdots + \lambda(x_k) \big\rangle.  
		\end{equation}
		\item [$(d)$] $F:=\mathrm{ran}\,\lambda$ is a convex cone in $\W$. It is closed if $\V$ is finite dimensional. 
		\item [$(e)$] The following are equivalent:
		\begin{itemize}
			\item [$(i)$] $x$ and $y$ commute in $(\V,\W,\lambda)$, that is, $\ip{x}{y} = \ip{\lx}{\ly}$.
			\item [$(ii)$] $\lambda(x+y)=\lx+\ly$.
			\item [$(iii)$] $\norm{\lx - \ly} = \norm{x-y}$.
		\end{itemize}
	\end{itemize}
\end{theorem}

Observe that Item $(b)$ above gives the continuity of $\lambda$; thus, every $\lambda$-orbit is closed in $\V$, and, additionally compact when $\V$ is finite dimensional.

\begin{proposition}\label{image of a convex set}
	Consider a FTvN system $(\V,\W,\lambda)$. If $E$ is convex and spectral in $\V$, then $\lambda(E)$ is convex in $\W$. In particular, if $E$ is a spectral set that is also a convex cone, then $\lambda(E)$ is a convex cone.
\end{proposition}

\begin{proof}
	Assume that $E$ is convex and spectral in $\V$. Let $\lu, \lv \in \lambda(E)$, where $u, v \in E$. We show that $t\lu + (1-t) \lv \in \lambda(E)$ for any $0 \leq t \leq 1$. To show this, we apply $(A3)$ in Definition \ref{ftvn} with $q = \lu$ and $c=v$ to get an $x \in \V$ such that
	\[ \lx = q = \lu \quad \text{and} \quad \ip{x}{v} = \ip{\lx}{\lv}. \]
	As $x$ and $v$ commute and  $\lambda$ is positively homogeneous, $tx$ and $(1-t)v$ commute as well. By Theorem \ref{basic theorem}(e), $\lambda \big( tx + (1-t)v \big) = \lambda(tx) + \lambda \big( (1-t)v \big)$. Since $\lx = \lu$ and $E$ is spectral, $x \in [u]\subseteq E$. Also, as $x, v\in E$, because of convexity, $tx + (1-t)v \in E$. Hence,
	\[ t \lu + (1-t)\lv = \lambda(tx) + \lambda \big( (1-t)v \big) = \lambda \big( tx + (1-t)v \big) \in \lambda(E). \]
	Thus, $\lambda(E)$ is convex.
	The last asserted statement follows from the positive homogeneity of $\lambda$.
\end{proof}

The above result together with the equality $E=\lambda^{-1}(\lambda(E))$ shows that 
	 every convex spectral set in $\V$ can be written as the $\lambda$-inverse of a convex set in $\W$.

\begin{theorem} \label{basic theorem2} 
	Consider a FTvN system $(\V,\W,\lambda)$. Let $E$ be a spectral set in $\V$, $d\in \V$, and $\alpha\in \R$. Then the  following are equivalent:
	\begin{itemize}
		\item [$(i)$] $\alpha \geq \ip{d}{z} \,\, \text{for all} \,\, z \in E$.
		\item [$(ii)$] $\alpha \geq \ip{d}{z} \,\, \text{for all} \,\, z \in \overline{\conv(E)}$.
		\item [$(iii)$] $\alpha\geq \ip{\lambda(d)}{\lambda(z)} \,\, \text{for all} \,\, z\in E$.
		\item [$(iv)$] $\alpha\geq \ip{\lambda(d)}{\lambda(z)} \,\, \text{for all} \,\, z \in \overline{\conv(E)}$.
	\end{itemize}
\end{theorem}

\begin{proof}
	$(i) \Leftrightarrow (ii)$: The implication $(i) \Rightarrow (ii)$ comes from the bilinearity and continuity of the inner product. The reverse implication is obvious. 
	
	\smallgap
	
	$(i) \Leftrightarrow (iii)$: When $(i)$ holds, $\alpha \geq \ip{d}{x}$ for all $x \in [z]$, where $z\in E$. Taking the maximum over $x$ in $[z]$ and using (\ref{intro ftvn}), we get $(iii)$. The reverse implication comes from the inequality $\ip{\lambda(d)}{\lambda(z)} \geq \ip{d}{z}$.
	
	\smallgap
	
	$(i) \Leftrightarrow (iv)$: Assume that $\alpha\geq \ip{d}{x}$ for all $x\in E$. Let $y\in \conv(E)$ so that $y=\sum_{i=1}^{k} t_ix_i$, where $t_i$s are nonnegative numbers adding up to one and $x_i$s belong to $E$. From $(iii)$, we get
	\[ \alpha \geq \ip{\lambda(d)}{\lambda(x_i)} \,\, \text{for all} \,\, i = 1, 2, \ldots, k; \]
	hence,
	\[ \alpha\geq \Big\langle \lambda(d),\, \sum_{i=1}^{k}t_i\lambda(x_i) \Big\rangle. \]
	Now, by (\ref{general sublinearity}) and the positive homogeneity of $\lambda$, we have
	\[ \alpha \geq \Big\langle \lambda(d),\, \sum_{i=1}^{k}\lambda(t_ix_i)\Big \rangle \geq \Big \langle \lambda(d), \lambda \Big( \sum_{i=1}^{k} t_ix_i \Big) \Big \rangle = \ip{\lambda(d)}{\ly}. \]
	As this inequality holds for all $y\in \conv(E)$, by the continuity of $\lambda$, we have 
	\[ \alpha \geq \ip{\lambda(d)}{\lambda(z)} \,\, \text{for all} \,\, z \in \overline{\conv(E)}. \]
	Thus we have $(iv)$. The reverse implication follows from the inequality $\ip{\lambda(d)}{\lambda(z)} \geq \ip{d}{z}$. This completes the proof.
\end{proof}

The implications $(i)\Leftrightarrow (iii)$ and $(i)\Leftrightarrow (iv)$ in the above theorem show the advantage of working with spectral sets: A linear optimization problem over a spectral set in $\V$ can be reformulated as a similar problem in $\W$.  Specifically, {\it for any vector $d$ and a spectral set $E$ in $\V$, the maximum (or supremum) of the function $z \mapsto \ip{d}{z}$ over  $E$ is the same as the maximum (respectively, supremum) of the function $w \mapsto \ip{\lambda(d)}{w}$ over the set $\lambda(E)$ or the set $\lambda\big(\, \overline{\conv(E)} \,\big)$. } Some similar statements can be made for distance and convex functions, see \cite{gowda-ftvn}.

\gap

The following result shows that the property of being spectral is invariant under certain topological and algebraic operations. Recall that, for a set $S$ in $\V$,  $S^*$ and $S^p$ denote, respectively, the {\it dual} and {\it polar cones} of $S$. We note that $S^p=-(S^*)$ and so, if $S$ is a closed convex cone, then $S^{pp}=S$. 

\begin{comment}
The following result shows that the property of being spectral is invariant under certain topological and algebraic operations. Recall that, for a set $S$ in $\V$,  we write $\overline{S}$, $S^\circ$, $\partial(S)$, and $S^\perp$ for the {\it closure, interior, boundary, and orthogonal complement} of $S$, respectively; $\conv(S)$ denotes the {\it convex hull} of $S$. Also, $S^*$ and $S^p$ denote, respectively, the dual and polar cones of $S$. We note that $S^p=-(S^*)$ and so, if $S$ is a closed convex cone, then $S^{pp}=S$. 
\end{comment}

\begin{proposition} \label{spectral invariance prop} Let $E$ be a spectral set in a FTvN system $(\V, \W, \lambda)$. Then the following statements hold:
	\begin{itemize}
		\item [$(a)$] $\overline{E}$, $E^\circ$, and $\partial(E)$ are  spectral.
		\item [$(b)$] If $\V$ is a Hilbert space, then $\overline{\conv(E)}$ is a spectral set.
		\item [$(c)$] If $\V$ is finite dimensional, then $\conv(E)$ is a spectral set; additionally, the convex cone generated by $E$ is also spectral.
		\item [$(d)$] If $\V$ is a Hilbert space, then $E^p$ is a spectral set. In particular, if $\V$ is a Hilbert space and $S$ is a spectral set which is also a subspace in $\V$, then, $S^\perp$ is spectral. 
		\item [$(e)$] If $\V$ is a Hilbert space, then the sum of two compact convex spectral sets in $\V$ is spectral. 
		\item [$(f)$] If $\V$ is finite dimensional, then the sum of two convex spectral sets is spectral.
	\end{itemize}
\end{proposition}

\gap

\begin{proof}
	$(a)$ First consider the closure of $E$. Let $x\in \overline{E}$ and $y\in [x]$. We need to show that $y\in \overline{E}$. As $x\in \overline{E}$, there exists a sequence $(x_k)$ in $E$ such that $x_k \to x$. For each $k$, corresponding to $c:=y$ and $q:=\lambda(x_k)$ in Definition \ref{ftvn}, there exists $y_k$ that commutes with $y$ and $\lambda(y_k)=\lambda(x_k)$. We then have
	\[ \norm{y_k - y} = \norm{\lambda(y_k) - \ly} = \norm{\lambda(x_k) - \lx} \leq \norm{x_k - x}, \]
	where the first equality is due to Item $(e)(iii)$ and the (last) inequality follows from Item $(b)$ in Theorem \ref{basic theorem}.
	Since $x_k \to x$, we see that $y_k \to y$, where $y_k\in [x_k]\subseteq E$ (the inclusion is due to $E$ being spectral).
	Thus, $y\in \overline{E}$. Hence, $\overline{E}$ is spectral.\\
	Since the relation $x\sim y$ is an equivalence relation, every spectral set is a union of $\lambda$-orbits; thus the complement of a spectral set is also a spectral set. Hence, $E^c$ is a spectral set. By what has been proved, $\overline{E^c}$ is spectral; thus, $(\overline{E^c})^c$ is also spectral. But the latter set is $E^\circ$ (the interior of $E$). Hence $E^\circ$ is spectral.\\
	Finally, $\partial(E)$, being the intersection of two spectral sets $\overline{E}$ and $(E^\circ)^c$, is also spectral.
	
	\smallgap
	
	$(b)$ Suppose $\V$ is a Hilbert space. Let $x\in \overline{\conv(E)}$. We show that $[x]\subseteq \overline{\conv(E)}$. Suppose, if possible, there is a $y\in [x]$ such that $y\not\in \overline{\conv(E)}$. Then, by the strong separation theorem (see Theorem \ref{strong separation theorem}), there exist $d \in \V$ and $\alpha \in \R$ such that
	\[ \ip{d}{y} > \alpha \geq \ip{d}{z} \,\, \text{for all} \,\, z \in \overline{\conv(E)}, \]
	and, in particular, 
	\[ \ip{d}{y} > \alpha \geq \ip{d}{z} \,\, \text{for all} \,\, z \in E. \]
	This implies, from Theorem \ref{basic theorem2}, 
	\[ \ip{d}{y} > \alpha \geq \ip{\lambda(d)}{\lambda(z)} \,\, \text{for all}\, \,z\in \overline{\conv(E)}.\]
	In particular, we have $\ip{d}{y} > \ip{\lambda(d)}{\lx}$. Since $\lx=\ly$, this implies
	\[ \ip{\lambda(d)}{\ly} \geq \ip{d}{y} > \ip{\lambda(d)}{\ly}, \]
	which is clearly a contradiction. Hence, $[x]\subseteq \overline{\conv(E)}$, proving $(b)$.
	
	\smallgap
	
	$(c)$ Suppose $\V$ is finite dimensional and let $x\in \conv(E)$. We show that $[x]\subseteq \conv(E)$. As $x\in \conv(E)$, we can write $x$ as a convex combination of elements $x_1,x_2,\ldots, x_N$ in $E$. Then the set $P:=\bigcup_{i=1}^{N} [x_i]$ is a spectral set which is contained in $E$. Since $\V$ is now assumed to be finite dimensional, each $\lambda$-orbit $[x_i]$ is compact; hence so is $P$.
	Since $\V$ is finite dimensional, we see that $\conv(P)$ is compact (see Theorem \ref{minkowski}) and, in particular, closed. Thus, by Item $(b)$ applied to the set $P$, $\conv(P)$ is a spectral set. This means,
	\[ [x] \subseteq [\conv(P)] = \conv(P) \subseteq \conv(E). \]
	This proves the spectrality of $\conv(E)$. Finally, due to the positive homogeneity of $\lambda$, the cone generated by $\conv(E)$ is also spectral. Thus we have Item $(c)$.
	
	\smallgap
	
	$(d)$ Since $[tx]=t[x]$ for all $t\geq 0$ in $\R$ and $x\in \V$, we see, by $(b)$, that closed (convex) conic hull of $E$ is also spectral. Since the polar of a set is the same as the polar of its closed conic hull, we assume without loss of generality that $E$ is a closed convex cone. We show that (the closed convex cone) $E^p$ is also spectral. Let $x\in E^p$ and $y\in \V$ with $\ly=\lx$. We claim that $y\in E^p$. If this were not true, then by the strong separation theorem (see Theorem \ref{strong separation theorem}), there would exist a nonzero $d$ in $\V$ such that 
	\[ \ip{d}{y} > 0 \geq \ip{d}{u} \,\, \text{for all} \,\, u \in E^p. \]
	It follows that $d\in E^{pp}=E$ and hence $[d]\subseteq E$. Since $x\in E^p$, we have
	\[ 0 \geq \ip{z}{x} \,\, \text{for all} \,\, z \in [d]. \]
	Taking the maximum over $z$, this results in $0\geq \ip{\lambda(d)}{\lx}$; hence,
	\[ \ip{d}{y} > 0 \geq \ip{\lambda(d)}{\lx} = \ip{\lambda(d)}{\ly}, \]
	which is clearly a contradiction. Hence, $y\in E^p$, proving the spectrality of $E^p$. Now suppose $S$ is a spectral set that is also a subspace. Then, $S^\perp=S^p$ is spectral.
	
	\smallgap
	
	$(e)$ Suppose that $\V$ is a Hilbert space and $E_1$ and $E_2$ are  compact, convex, and spectral in $\V$. As $E_1+E_2$ is convex, we show that $E_1+E_2$ is spectral. Let  $x\in \V$ and  $u\in E_1+E_2$ with $\lx=\lu$. We need to show that $x\in E_1+E_2$.
	Suppose, if possible, $x\not \in E_1+E_2$. Since $E_1+E_2$ is compact and convex (and $\V$ is a Hilbert space), by the strong separation theorem (see Theorem \ref{strong separation theorem}), there exist $c\in \V$ and $\alpha\in \R$ such that 
	\[ \ip{c}{x} > \alpha \geq \ip{c}{y_1+y_2} = \ip{c}{y_1} + \ip{c}{y_2} \,\, \text{for all} \,\, y_1 \in E_1, y_2 \in E_2. \]
	Now, writing $u = u_1 + u_2$, where $u_1\in E_1$ and $u_2\in E_2$, we vary $y_1$ over $[u_1]$ (which is a subset of $E_1$) and $y_2$ over $[u_2]$ (a subset of $E_2$). Applying (\ref{intro ftvn}), this results in
	\[ \ip{c}{x} > \alpha \geq \ip{\lc}{\lambda(u_1)} + \ip{\lc}{\lambda(u_2)} = \ip{\lc}{\lambda(u_1)+\lambda(u_2)} \]
	and
	\[ \ip{\lc}{\lx} > \alpha \geq \ip{\lc}{\lambda(u_1) + \lambda(u_2)}.\] 
	As $\lx = \lu = \lambda(u_1 + u_2)$, we see that
	\[ \ip{\lc}{\lambda(u_1 + u_2)} > \alpha \geq \ip{\lc}{\lambda(u_1) + \lambda(u_2)}, \]
	contradicting Theorem \ref{basic theorem}(c). Hence, $x \in E_1+E_2$, proving the spectrality of the sum.
	
	\smallgap
	
	$(f)$ Suppose $\V$ is finite dimensional with $E_1$ and $E_2$ convex and spectral. As before, let $x\in \V$, $u\in E_1 + E_2$, $\lx = \lu$. Let $u = u_1 + u_2$, where $u_1\in E_1$ and $u_2\in E_2$. Now, for $i=1,2$, $[u_i]$ (which is a subset of $E_i$) is spectral and, since $\V$ is finite dimensional, compact.
	By Item $(c)$ above, $\conv\,[u_i]$ is compact, convex, and spectral. By our previous case, $\conv\,[u_1] + \conv\,[u_2]$ is spectral. Since $u=u_1+u_2\in \conv\,[u_1] + \conv\,[u_2]$ and $\lx=\lu$, we see that 
	\[ x\in \conv\,[u_1] + \conv\,[u_2]\subseteq E_1+E_2, \]
	where the inclusion comes from the convexity of sets $E_1$ and $E_2$. Thus, $E_1+E_2$ is spectral.
\end{proof}

\begin{corollary}
	Consider a FTvN system $(\V, \W, \lambda)$, where $\V$ is finite dimensional. Then, for all $a, b \in \V$,
	\[ \conv\,[a+b] \subseteq \conv\,[a]+\conv\,[b]. \]
\end{corollary}

\begin{proof}
	As $\V$ is finite dimensional, the orbits $[a]$, $[b]$, and $[a+b]$ are compact; hence their convex hulls are also compact. Suppose $u \in \conv\,[a+b]$ but $u \notin \conv\,[a] + \conv\,[b]$. As the set $\conv\,[a] + \conv\,[b]$ is compact and convex, by the strong separation theorem, there exist $c \in \V$ and $\alpha \in \R$ such that 
	\[ \ip{c}{u} > \alpha \geq \ip{c}{x} + \ip{c}{y} \,\, \text{for all} \,\, x \in [a], y \in [b]. \]
	By applying (\ref{intro ftvn}) and using Theorem \ref{basic theorem} (c), we get
	\[ \ip{c}{u} > \alpha \geq \ip{\lc}{\lambda(a)} + \ip{\lc}{\lambda(b)} \geq \ip{\lc}{\lambda(a+b)}. \]
	However, by specializing Theorem \ref{basic theorem2}, we have
	\[ \ip{\lc}{\lambda(a+b)} \geq \ip{\lc}{z} \,\, \text{for all} \,\, z \in \conv\,[a+b]. \]
	As $u\in \conv\,[a+b]$, we see that
	\[ \ip{c}{u} > \alpha \geq \ip{\lc}{\lambda(a+b)} \geq \ip{c}{u}, \]
	thus reaching a contradiction. Hence, the stated inclusion follows.
\end{proof}

\textbf{Remark.} It is known, see \cite{jeong-gowda-spectral cone}, that if $E$ is a spectral set in a Euclidean Jordan algebra, then the sets $-E$ and $E^*$ (the dual of $E$) are also spectral. A similar statement can easily be verified in the setting of a normal decomposition system. However, we do not know if these hold in general FTvN systems. Hence, we do not know if we can go from (the spectrality of) $E^p$ to $E^*:=-E^p$.

%%%%%%%%%%%%%%%%%%%%%%%%%%%%%%%%%%%%%%%%%%%%%%%%%%%%%%%%%%%%%%%
\section{The center of a FTvN system}
In this section, we make an in-depth study of commutativity property. We begin with a definition.

\begin{definition}
	Let $(\V, \W, \lambda)$ be a FTvN system. For any $x\in \V$, let 
	\[ C(x) := \{y\in \V : \text{$y$ commutes with $x$} \}. \]
	The set $C := \bigcap_{x\in \V} C(x)$ is called the center of $(\V, \W, \lambda)$.
\end{definition}

To illustrate, we consider Example \ref{rn} with $n\geq 2$. If $e_1, e_2, \ldots, e_n$ are the standard coordinate vectors in $\Rn$ (so $e_k$ has $1$ in its $k$th slot and zeros elsewhere), then $y = (y_1, y_2, \ldots, y_n) \in C(e_k)$ if and only if $y_k \geq y_i$ for all $i = 1, 2, \ldots, n$. It is easy to see that $C = \R\,e$, where $e$ is the vector of ones in $\Rn$.
We specifically note that $e_2 \notin C(e_1)$ while $e_2 \in [e_1]$. Thus, $C(e_1)$ is not a spectral set.

\begin{proposition} \label{prop: subspace}
	In a FTvN system $(\V, \W, \lambda)$, the following hold:
	\begin{itemize}
		\item [$(a)$] For every $x\in \V$, $C(x)$ is a closed convex cone. Also, $[x] \cap C(x) = \{x\}$.
		\item [$(b)$] If $\lx = -\ly$ for some $x, y\in \V$, then $x = -y$ and $x, y\in C$.
		\item [$(c)$] $C$ is a closed (linear) subspace of $\V$ and $\lambda$ is linear on it.
		\item [$(d)$] $C = \{x \in \V : \lambda(-x) = -\lx\}$.
		\item [$(e)$] If $x$ and $-x$ commute, then $x \in C$. Consequently, $C = \{x \in \V : \text{$x$ and $-x$ commute}\}$.
	\end{itemize}
\end{proposition}

\begin{proof}
	$(a)$ Fix $x\in \V$. Since $\lambda$ is continuous and positively homogeneous (see Theorem \ref{basic theorem}), $C(x)$ is a closed cone. We now prove that it is convex. Let $u,v\in C(x)$. Then, for any $z\in [x]$,
	\begin{align*}
		\ip{z}{u+v} & = \ip{z}{u} + \ip{z}{v}  \\
		& \leq \ip{\lambda(z)}{\lu} + \ip{\lambda(z)}{\lv} \\
		& = \ip{\lx}{\lu} + \ip{\lx}{\lv} \\
		& = \ip{x}{u} + \ip{x}{v}.                                       
	\end{align*} 
	Taking the maximum over $z$ in $[x]$ we get $\ip{\lx}{\lambda(u+v)} \leq \ip{x}{u+v}$. As the reverse inequality is obvious, we get the equality $\ip{x}{u+v} = \ip{\lx}{\lambda(u+v)}$. This proves that $u+v\in C(x)$. Hence, $C(x)$ is a closed convex cone.
	
	Now we show that $[x] \cap C(x) = \{x\}$. As $\lambda$ is norm-preserving, $x$ commutes with itself;  so
	$x\in [x]\cap C(x)$. If $y\in [x]\cap C(x)$, then $y$ commutes with $x$ and $\ly=\lx$. From Theorem \ref{basic theorem}(e),
	\[ \norm{x-y} = \norm{\lx - \ly} = 0. \]
	Thus, $y=x$. 
	
	\smallgap
	
	$(b)$ Let $\lx=-\ly$. Then, $\lx+\ly=0$. Applying Theorem \ref{basic theorem}, for any $c\in \V$ we have  
	$\ip{\lc}{\lambda(x+y)} \leq \ip{\lc}{\lx + \ly} = 0$. Specializing this to $c = x+y$, we get $\lambda(x+y)=0$, hence $x+y=0$ (as $\lambda$ is norm-preserving). This shows that $x=-y$. Now we show that $x\in C$. For any $z\in \V$, we have
	\[ \ip{x}{z} \leq \ip{\lx}{\lambda(z)} = \ip{-\lambda(-x)}{\lambda(z)} = -\ip{\lambda(-x)}{\lambda(z)} \leq -\ip{-x}{z}=\ip{x}{z}. \]
	This shows that $\ip{x}{z} = \ip{\lx}{\lambda(z)}$. So, $x$ commutes with (every) $z\in \V$; hence $x\in C$. similarly, $y\in C$.
	
	\smallgap
	
	$(c)$ Clearly $C$, being the intersection of closed convex cones, is also a closed convex cone. To show that it is a subspace, we show that $x\in C$ implies $-x\in C$. Suppose $x\in C$. Then, $x$ and $-x$ commute; hence by Item $(e)$ of Theorem \ref{basic theorem}, $0 = \lambda(0) = \lambda(x-x) = \lx + \lambda(-x)$. This implies that $\lambda(-x) = -\lx$. By $(b)$, $-x\in C$. Hence, $C$ is a subspace.
	
	Now, as observed above, for every $x\in C$, $\lambda(-x) = -\lx$. Since $\lambda$ is additive on $C$ (from  Theorem \ref{basic theorem}, Item $(e)$) and positively homogeneous, we see that $\lambda$ is linear on $C$. 
	
	\smallgap
	
	$(d)$ If $x \in C$, by the linearity of $\lambda$ on $C$, we have $\lambda(-x) = -\lx$. On the other hand, if $\lambda(-x) = -\lx$, then by $(b)$, $x \in C$. Thus we have the stated equality.
	
	\smallgap
	
	$(e)$ Suppose $x$ and $-x$ commute. From Theorem \ref{basic theorem} $(e)$, we see that
	\[ 0 = \lambda(0) = \lambda \big( x + (-x) \big) = \lambda(x) + \lambda(-x). \]
	It follows that $\lambda(-x) = -\lambda(x)$; hence $x \in C$ by $(d)$.
\end{proof}

\begin{proposition}\label{e-orbit prop}
	Let $(\V,\W,\lambda)$ be a FTvN system. Then, 
	\[ C = \big\{ u \in \V : [u] = \{u\} \big\}. \]
\end{proposition}

\begin{proof}
	Clearly, $0\in C$ and $[0]=\{0\}$. We show that $u\in C$ if and only if $[u]=\{u\}$ by assuming $u\neq 0$. First, suppose $[u]=\{u\}$. Then, for any $v\in \V$, maximizing $\ip{v}{x}$ over the singleton set $[u]$ and using (\ref{intro ftvn}), we see that $\ip{v}{u} = \ip{\lv}{\lu}$. This proves that $u$ commutes with every  $v\in \V$; hence $u\in C$. 
	
	Conversely, suppose $u \in C$. If $v \in [u]$, then $v \in [u] \cap C(u)$. Since $[u]\cap C(u)= \{u\}$ by Proposition \ref{prop: subspace}(a), we see that $v = u$, proving $[u] = \{u\}$.
\end{proof}

\gap

The following remark summarizes various results characterizing elements of the center.

\gap

\textbf{Remark.} Let $(\V, \W, \lambda)$ be a FTvN system. Then the following are equivalent for an $x\in \V$:
\begin{itemize}
    \item [$(i)$] $x \in C$.
    \item [$(ii)$] $[x] = \{x\}$.
    \item [$(iii)$] $x$ and $-x$ commute.
    \item [$(iv)$] $\lambda(-x) = -\lambda(x)$.
    \item [$(v)$] $[-x] = \{-x\}$.
    \item [$(vi)$] $-x \in C$.
\end{itemize}

\gap

From the previous results, we know that $\lambda(C)$ is a subspace of $\W$ and $\lambda(\V)$ (the range of $\lambda$) is a convex cone. Our next result relates these two sets and characterizes when  the range of $\lambda$ can be pointed or a subspace.

\begin{corollary}\label{lineality result}
	Let $(\V,\W,\lambda)$ be a FTvN system. Then $\lambda(C)$ is the lineality space of the convex cone $\lambda(\V)$, that is,
	\[ \lambda(C) = \lambda(\V) \cap -\lambda(\V). \]
	Consequently,
	\begin{itemize}
		\item [$(a)$] $\lambda(\V)$ is pointed if and only if $C = \{0\}$, and
		\item [$(b)$] $\lambda(\V)$ is a subspace (of $\W)$ if and only if $C = \V$.
	\end{itemize}
\end{corollary}

\begin{proof}
	We first prove the equality $\lambda(C) = \lambda(\V) \cap -\lambda(\V)$. If $q \in \lambda(C)$, then $q = \lu$ for some $u \in C$. By the linearity of $\lambda$ on $C$, $q = \lu = -\lambda(-u)$; hence $q \in \lambda(\V)\cap -\lambda(\V)$. Conversely, if $q \in \lambda(\V) \cap - \lambda(\V)$, then $q = \lx = -\ly$ for some $x, y \in \V$; by Proposition \ref{prop: subspace}, $x \in C$. Thus, $q = \lx \in \lambda(C)$.
	
	\smallgap
	
	$(a)$ From the above equality, $\lambda(\V)$ is  pointed if and only if $\lambda(C) = \{0\}$. As $\lambda$ is norm-preserving, this can hold if and only if $C = \{0\}$.
	
	\smallgap
	
	$(b)$ Suppose $\lambda(\V)$  is a subspace of $\W$. Then, by the above equality, $\lambda(C) = \lambda(\V)$. Then, for any $x \in \V$, there is a $u \in C$ such that $\lx = \lu$. As $[u] = \{u\}$ from the previous result, we see that $x = u \in C$. Thus, $\V \subseteq C$ proving $\V = C$. On the other hand, if $\V = C$, then, by the linearity of $\lambda$ on $C$, $\lambda(\V) = \lambda(C)$ is a subspace of $\W$.
\end{proof}

\begin{proposition}
	Let $(\V,\W,\lambda)$ be a FTvN system. Then, 
	\begin{itemize}
		\item [$(a)$] Every subset of $C$ is a spectral set. 
		\item [$(b)$] The orthogonal complement $C^\perp$ of $C$ in $\V$ is a spectral set.
	\end{itemize}
\end{proposition}

\begin{proof}
	$(a)$ Let $D \subseteq C$. Then, for any $x \in D$, we see that $[x] = \{x\}$ by Proposition \ref{e-orbit prop}. It follows that $[x] = \{x\} \subseteq D$. Thus, $D$ is a spectral set.
	
	\smallgap
	
	$(b)$ Let $v\in C^\perp$ and $x\in [v]$. Then, for all $c\in C$,
	\[ \ip{x}{c} = \ip{\lx}{\lc} = \ip{\lv}{\lc} = \ip{v}{c} = 0. \]
	Hence, $x\in C^\perp$, that is, $[v]\subseteq C^\perp$. It follows that $C^\perp$ is a spectral set.
\end{proof}

\begin{theorem} (Decomposition Theorem)
	Let $(\V,\W,\lambda)$ be a FTvN system, where $\V$ is a Hilbert space. With $C$ denoting the center, we have the following:
	\begin{itemize}
		\item [$(a)$] $\V = C + C^\perp$.
		\item [$(b)$] $(C,\W,\lambda)$ is a FTvN system whose center is $C$ (so any two elements in this system commute).
		\item [$(c)$] $(C^\perp, \W,\lambda)$ is a FTvN system whose center is $\{0\}$.
	\end{itemize} 
\end{theorem}

\begin{proof}
	$(a)$ As $\V$ is a Hilbert space and $C$ is a closed subspace of $\V$, this is clear.
	
	\smallgap
	
	$(b)$ Note that $C$, being a closed subspace of $\V$, is a Hilbert space and $\lambda$ restricted to $C$ is norm-preserving. Also, for any $u\in C$, $[u]\subseteq C$ (from the previous result). The defining properties of the FTvN system $(\V,\W,\lambda)$ carry over to that of $(C,\W,\lambda)$. Since any two elements in $C$ commute, the center of this new system is $C$.
	
	\smallgap
	
	$(c)$ As in Item $(b)$, we easily verify that $(C^\perp, \W,\lambda)$ is a FTvN system. We now describe its center. Suppose $d\in C^\perp$ commutes with every element in $C^\perp$, that is, $\ip{d}{x} = \ip{\lambda(d)}{\lx}$ for all $x \in C^\perp$. Now, take $x\in C^\perp$ and $y\in C$. Then, we have $\ip{d}{y} = \ip{\lambda(d)}{\ly}$, and thus 
	\begin{align*}
		\ip{d}{x+y} & = \ip{d}{x} + \ip{d}{y}                       \\
		& = \ip{\lambda(d)}{\lx} + \ip{\lambda(d)}{\ly} \\
		& = \ip{\lambda(d)}{\lx + \ly}                  \\
		& \geq \ip{\lambda(d)}{\lambda(x+y)},           
	\end{align*}
	where the inequality comes from Item $(c)$ in Theorem \ref{basic theorem}. Since the reverse inequality always holds, we have the equality $\ip{d}{x+y} = \ip{\lambda(d)}{\lambda(x+y)}$. As $x+y$ is an arbitrary element of $\V$ (from Item $(a)$), $d$ commutes with every element in $\V$; hence belongs to $C$. As $d\in C^\perp$, we see that $d=0$. Thus, the center of $(C^\perp, \W,\lambda)$ is $\{0\}$.
\end{proof}

The above result allows us to write any FTvN system $(\V,\W,\lambda)$ with $\V$ a Hilbert space as a Cartesian product of two FTvN systems where one has a `full center' and the other has a `trivial center'. We elaborate this as follows. Consider a system as in the above theorem.
Since $\lambda$ is linear on $C$, the image $\W_1:=\lambda(C)$ is a subspace of $\W$. Since $(C^\perp, \W,\lambda)$ is a FTvN system, by Theorem \ref{basic theorem} $(d)$, we see that $\lambda(C^\perp)$ is a convex cone in $\W$, hence $\W_2:=\lambda(C^\perp)-\lambda(C^\perp)$ is a subspace of $\W$. Furthermore, for all $c\in C$ and $d\in C^\perp$, $0 = \ip{c}{d} = \ip{\lc}{\lambda(d)}$. Thus, $\W_1 \perp \W_2$ in $\W$.
Now, let $\V_1 := C$, $\V_2 := C^\perp$, and $\lambda_1$ and $\lambda_2$ denote the restriction of $\lambda$ to $\V_1$ and $\V_2$ respectively. Then, $(\V, \W, \lambda)$ is the Cartesian product of $(\V_1, \W_1, \lambda_1)$ and $(\V_2, \W_2, \lambda_2)$,
where the center of the first system is all of $\V_1$, that is the system has a `full center' and in the second system, the center is $\{0\}$, that is, the system has a `trivial center'. 

\gap

Before giving examples, we introduce a definition.

\begin{definition}
	Let $(\V,\W,\lambda)$ be a FTvN system. An element $e\in \V$ is called a unit element if it is nonzero and $C=\R\,e$.
\end{definition}

A unit element in a FTvN system, if it exists, is unique up to a scalar. Moreover, by Proposition \ref{e-orbit prop}, {\it if $e$ is a unit element, then $[e]=\{e\}$.} A general FTvN system may not have a unit element. However, when the system has a trivial center, we can adjoin a unit element as follows:
Suppose the center of $(\V,\W,\lambda)$ is $\{0\}$. Then, in the product space
$(\R\times \V, \R\times \W,\mu)$, where $\mu(t,x):=(t,\lx)$,
$e:=(1,0)$ is a unit element. 

\gap

\noindent{\bf Example \ref{ips}} ({\it Real inner product space, continued\,}) In the FTvN system $(\V, \R, \lambda)$, where $\V$ is an inner product space and $\lx := \norm{x}$, by the equality case in Cauchy-Schwarz inequality, we see that $C(x)=\R_+\,x$ for all nonzero $x\in \V$ and $C(0)=\V$. Hence, in this setting, $C=\{0\}$.

\gap

\noindent{\bf Example \ref{discrete space}} ({\it Discrete system, continued\,}) Consider the FTvN system $(\V,\V,S)$, where $\lambda=S$ is a linear isometry. In this case, any two elements of $\V$ commute (and, as noted before, every $\lambda$-orbit is a singleton); hence $C=\V$. Conversely, if $(\V,\W,\lambda)$ is a FTvN system with $C=\V$, then (thanks to Items $(a)$ and $(e) \, (ii)$ in Theorem \ref{basic theorem}), $\lambda$ is a linear isometry. Hence, as long as $\dim(\V) \geq 2$, this system does not have a unit element.

\gap

\noindent{\bf Example \ref{space mn}} ({\it The space $M_n$, continued\,})
Consider the FTvN system $(M_n,M_n,\lambda)$, where $M_n$ denotes the set of all $n\times n$ complex matrices and the map $\lambda$ takes $X$ in $M_n$ to $\diag s(X)$ (the diagonal matrix consisting of the singular values of $X$ written in the decreasing order). In this system, if $X$ is in the center, then $X$ and $-X$ commute. By Item $(e)(ii)$ in Theorem \ref{basic theorem}, $\lambda(X)+\lambda(-X)=\lambda(0)=0$. As the singular values of a matrix are always nonnegative, we see that $\lambda(X) = 0$ and (hence) $X=0$. So, in this case, $C=\{0\}$.

\gap

\noindent{\bf Example \ref{eja}} ({\it Euclidean Jordan algebras, continued\,}) Let $e$ denote the unit element in the Euclidean Jordan algebra $\V$, that is, $x\circ e=x$ for all $x\in \V$. We claim that in the FTvN system $(\V,\Rn,\lambda)$, we have $C = \R\,e$.
Since $e=e_1+e_2+\cdots+e_n$ for every Jordan frame $\{e_1,e_2,\ldots,e_n\}$ in $\V$, any scalar multiple of $e$ strongly operator commutes with every element of $\V$ (see the definition in Example \ref{eja}); hence $\R\,e\subseteq C$. Conversely, suppose $x\in C$. Then, $x$ commutes with $-x$ in $(\V,\Rn,\lambda)$. By Theorem \ref{basic theorem}(e), $\lambda(x+(-1)x)=\lx+\lambda(-x)$ and so $0=\lx+\lambda(-x)$. We observe that entries of $\lx$ are in decreasing order while those of $-\lambda(-x)$ are in the increasing order. Since these two vectors are equal, all entries of $\lx$ must be equal. By the spectral theorem \cite{faraut-koranyi}, $x$ must be a multiple of $e$. Thus, $C\subseteq \R\,e$. We conclude that $C=\R\,e$.

\gap

\noindent{\bf Example \ref{hyperbolic}} ({\it Hyperbolic polynomials, continued\,})
Consider the FTvN system $(\V,\Rn,\lambda)$ that  arises via a  polynomial $p$ that is hyperbolic relative to $e\in \V$ and also complete and  isometric. We  claim that $C=\R\,e$. Consider any $x\in \V$. As $\lx$ consists of the roots of the  equation $p(te-x)=0$ and $\lambda(e)$ is the vector of $1$s in $\Rn$, we see that $\lambda(x+e)=\lx+\lambda(e)$. So, by Theorem \ref{basic theorem}(e), $x$  and $e$ commute. Thus, $\R\,e\subseteq C$. Now suppose $0\neq x\in C$. Then, $x$ and $-x$ commute; as in the previous example, $\lx=-\lambda(-x)$. Since the entries of $\lx$ are decreasing, we see that $\lx$ is a multiple of $\lambda(e)$; without loss of generality, let $\lx=\rho\,\lambda(e)$, where $\rho\geq 0$. Then, $\lambda(\rho\,e-x)=0$. As $p$ is complete, $x=\rho\,e$. Hence, $C\subseteq \R\,e$. Thus, $C=\R\,e$.

\gap

Motivated by the above example, we formulate the following result.

\begin{proposition}
	Consider a FTvN system $(\V,\W,\lambda)$, where $\dim(\V)\geq 2$. Then $C=\{0\}$ under one of the following conditions:
	\begin{itemize}
		\item [$(i)$] $\ip{\lx}{\ly} = 0 \Rightarrow x=0 \,\, \text{or} \,\, y=0$.
		\item [$(ii)$] $\big[ 0 = \ip{x}{y} = \ip{\lx}{\ly} \big] \Rightarrow x=0 \,\, \text{or} \,\, y=0$.
	\end{itemize}
\end{proposition}

\begin{proof}
	As $(i) \Rightarrow (ii)$, we prove the result assuming $(ii)$. 
	Suppose $0\neq x\in C$. As $\dim(\V)\geq 2$, there is a $y\neq 0$ in $\V$ that is orthogonal to $x$. Then, $0 = \ip{x}{y} = \ip{\lx}{\ly}$ (the second equality is due to $x$ being in the center). Since $x$ and $y$ are nonzero we reach a contradiction to $(ii)$. Thus $C=\{0\}$. 
\end{proof}

We note here that Examples \ref{discrete space}, \ref{eja with absolute value}, \ref{space mn}, and \ref{sequence space} satisfy condition $(a)$ of the above proposition. 

\gap

The weaker condition $(b)$ in the above result leads to an interesting proposition.

\begin{proposition}
 	Suppose $(\V,\W,\lambda)$ is a FTvN system where $\V$ is finite dimensional with $\dim(\V)\geq 2$ and 
	\[ \big[ 0 = \ip{x}{y} = \ip{\lx}{\ly} \big] \Rightarrow x=0 \,\, \text{or} \,\, y=0. \]
	Then, $\{0\}$ and $\V$ are the only  spectral cones in $\V$. Consequently, for any nonzero $u\in \V$, the convex cone generated by $[u]$ equals $\V$.
\end{proposition}

\begin{proof}
	Suppose, if possible, $K$ is a (nonempty)  spectral cone in $\V$ such that $\{0\} \neq K \neq \V$. We first assume that $K$ is closed. As $K$ is nonempty, closed, and different from $\V$, by the connectedness of $\V$, $\overline{K} (=K)\neq K^\circ$. So $\partial(K)=\overline{K}\,\setminus\,K^\circ\neq \emptyset$. We claim that $\partial(K) = \{0\}$. Suppose, if possible $0\neq x\in \partial(K)$. Then, as $\V$ is finite dimensional,
	by the supporting hyperplane theorem \cite{rockafellar}, there is a nonzero $d\in \V$ such that 
	\[ \ip{x}{d} \geq 0 \geq \ip{y}{d} \,\, \text{for all} \,\, y \in K. \]
	As $x\in K$, putting $y=x$ in the above inequality results in $0 \geq \ip{x}{d}$; so, $\ip{x}{d} =0$. Moreover, since $K$ is a spectral set, 
	\[ 0 \geq \max_{y\in [x]} \ip{d}{y} = \ip{\lx}{\lambda(d)} \geq \ip{x}{d} \geq 0. \]
	Hence, $0 = \ip{x}{d} = \ip{\lx}{\lambda(d)}$. Since $x$ and $d$ are nonzero, we reach a contradiction to the imposed condition. Thus, $\partial(K)=\{0\}$. This means that $K\setminus K^\circ =\{0\}$. As $K\neq \{0\}$, we see that $K^\circ \neq \emptyset$. Now, as $K\neq \V$, we let $z\in \V\,\setminus\, K$. Noting that $\dim(\V) \geq 2$, we take two linearly independent elements $u,v\in K^\circ$ and consider 
	line segments $[z,u]$ and $[z,v]$. Then, on the open segments $(z,u)$ and $(z,v)$, we must have points that are in $\partial(K)$. As $\partial(K)= \{0\}$, these boundary points must coincide with zero indicating that $u$ and $v$ are both nonzero multiples of $z$. This contradicts the linear independence of $u$ and $v$. Thus, we cannot have $\{0\}\neq K\neq \V$. So, $\{0\}$ and $\V$ are the only two closed spectral cones in $\V$.\\ Now suppose $K$ is a nonempty nonzero spectral cone. Then, $\overline{K}$ (the closure of $K$) is a nonempty nonzero closed spectral cone (the spectrality comes from Proposition \ref{spectral invariance prop}); so, $\overline{K}=\V$. However, in $\V$ (which is finite dimensional), if the closure of a convex set is $\V$, then the set must be equal to $\V$. (One can see this by working with relative interiors.)
	Hence, $K=\V$.\\
	Now suppose $u$ is any nonzero element in $\V$. Then,  the convex cone generated by $[u]$ is a nonzero spectral cone; by the above, this set is $\V$.
\end{proof}

As an illustration of the above result, consider the space
$M_n$ of Example \ref{space mn}. Then, for any nonzero $D\in M_n$, the set of all nonnegative linear combinations of matrices of the form $UDV$, where $U$ and $V$ are unitary matrices, is equal to $M_n$. 

%%%%%%%%%%%%%%%%%%%%%%%%%%%%%%%%%%%%%%%%%%%%%%%%%%%%%%%%%%%%%%%%%%%%%%%%%%%%%%
\section{Automorphisms}

\begin{definition}
	Let $(\V,\W,\lambda)$ be a FTvN system. An invertible linear transformation $A:\V \to \V$ is said to be an automorphism of the system if $\lambda(Ax)=\lx$ for all $x\in \V$. We denote the set of all automorphisms of $(\V,\W,\lambda)$ by $\Aut(\V,\W,\lambda)$, or simply by $\A$ when the context is clear.
\end{definition}

If $A$ is an automorphism of the FTvN system $(\V,\W,\lambda)$, then, $\norm{Ax} = \norm{\lambda(Ax)} = \norm{\lx} = \norm{x}$ for all $x$, so $A$ is norm/inner product preserving. Because $A$ is an invertible linear transformation, we
see that each automorphism of $\V$ is (by definition) an orthogonal transformation of $\V$. Let ${\cal O}(\V)$ denote the orthogonal group of $\V$ (thus forming a subset of the space of all bounded linear transformations on $\V$). It is easy to see that
\begin{center}
	{\it $\Aut(\V,\W,\lambda)$ is a closed subgroup of ${\cal O}(\V)$.}
\end{center}

The following statements are easy to see for any $A\in \Aut(\V,\W,\lambda)$.

\begin{itemize}
	\item If $x$ and $y$ commute in $(\V,\W,\lambda)$, then $Ax$ and $Ay$ commute.
	\item $A$ coincides with the  identity transformation on (the center) $C$, that is, $Ax = x$ for all $x \in C$ (see Proposition \ref{e-orbit prop}).
\end{itemize}

We now identify a few automorphism groups.

\gap

\noindent{\bf Example \ref{ips}} ({\it Real inner product space, continued\,}) 
In this system, $\Aut(\V,\W,\lambda)={\cal O}(\V)$.

\gap

\noindent{\bf Example \ref{discrete space}} ({\it Discrete space, continued\,}) 
Consider $(\V,\V,S)$, where $S$ is a linear isometry (not necessarily onto). If $A$ is an automorphism of this system, then $SAx=Sx$. Since $S$ is injective, we see that $Ax=x$ for all $x\in \V$; so, $A$ is the identity transformation. Hence in this setting, $\Aut(\V,\W,\lambda) = \{\mathrm{Id}\}$. This conclusion can also be seen by observing that on the center $C$ (which is $\V$ in this example), $A$ is the identity transformation.

\gap

\noindent{\bf Example \ref{rn}} ({\it The Euclidean space $\Rn$, continued\,}) 
Consider the FTvN systems $(\Rn,\Rn,\lambda)$ and $(\Rn,\Rn,\mu)$, where $\lx=x^\downarrow$ and $\mu(x)=|x|^\downarrow$. Then, $\Aut(\Rn, \Rn, \lambda)$ is the set of all $n \times n$ permutation matrices and $\Aut(\Rn, \Rn, \mu)$ is the set of all $n \times n$ signed permutation matrices (which are matrices of the form $DP$, where $D$ is a diagonal matrix whose diagonals are $\pm 1$ and $P$ is a permutation matrix).

\gap

\noindent{\bf Example \ref{eja}} ({\it Euclidean Jordan algebras, continued\,}) 
Here the FTvN system corresponds to a Euclidean Jordan algebra $\V$. Let $\Aut(\V)$ denote the set of all algebra automorphisms of $\V$. These are invertible linear transformations on $\V$ that satisfy the property
\[ A(x\circ y)=Ax\circ Ay \quad (x, y\in \V) \]
{\it We claim that the algebra automorphisms of $\V$ are the same as the automorphisms of the corresponding FTvN system, that is,}
\[ \Aut(\V,\Rn,\lambda) = \Aut(\V). \]
To see this, first suppose $A$ is an algebra automorphism of $\V$. Then, thanks to the spectral decomposition in $\V$, $A$ preserves eigenvalues, that is, $\lambda(Ax)=\lx$ for all $x\in \V$. Hence, by our definition, $A$ is an automorphism of the FTvN system $(\V,\Rn,\lambda)$. To see the converse, let $A$ be an automorphism of the system $(\V,\Rn,\lambda)$. Then $\lambda(Ax)=\lx$ for all $x$. So, $x \geq 0$ if and only if $Ax\geq 0$, where we write $x \geq 0$ to mean that all the eigenvalues of $x$ are nonnegative. Writing $\V_+:=\{x\in \V: x\geq 0\}$ (the symmetric cone of $\V$), we see that $A(\V_+)\subseteq \V_+$. Replacing $A$ by its inverse (which is also an automorphism), we see that $A^{-1}(\V_+)\subseteq \V_+$. Thus, $A(\V_+)=\V_+$. Since $A$ is also orthogonal and $\V$ carries the trace inner product, we conclude (see \cite{faraut-koranyi}, p.~57) that $A\in \Aut(\V)$, that is, $A$ is an algebra automorphism of $\V$.

\gap

\noindent{\bf Example \ref{nds}} ({\it Normal decomposition systems, continued\,}) 
We have observed that a normal decomposition system $(\V,\G,\gamma)$ can be regarded as a FTvN system $(\V,\W,\gamma)$, where $\W = \spn(\gamma(\V))$. Built into the definition of NDS is the condition $\gamma(Ax)=\gamma(x)$ for all $x\in \V$ and all $A\in \G$. It follows that
\[ \G\subseteq \Aut(\V,\W,\gamma). \]
The inclusion here can be proper. For example, if $\V$ is a simple Euclidean Jordan algebra carrying the trace inner product, we can let
$\G$ be the connected component of $\Aut(\V)$ containing the identity map and $\gamma$ be a suitable map (defined by means of a fixed Jordan frame, see \cite{gowda-jeong}). Then, we obtain the NDS $(\V,\G,\gamma)$ and the corresponding FTvN system $(\V,\W,\gamma)$ (which can be identified with $(\V,\Rn,\gamma)$). Clearly, $\G$ is a proper subset of $\Aut(\V,\W,\gamma)$.

\gap

\noindent{\bf Example \ref{sequence space}} ({\it The sequence space, continued\,})
Here, we describe all automorphisms of the FTvN system $\ell_2(\R)$. Consider the complete orthonormal system $\{e_1, e_2, \ldots, e_k,\ldots\}$ in $\ell_2(\R)$, where $e_k$ is the  sequence of zeros and ones with  $1$ appearing in the $k$th slot. Then, for any $x=(x_1, x_2, \ldots) \in \ell_2(\R)$, we can write $x = \sum_{k=1}^{\infty} x_ke_k$. Now, let $A$ be any automorphism of the FTvN system $\ell_2(\R)$ (so $A$ is an invertible linear transformation that preserves all $\lambda$-orbits). Then, for any $k$, $\lambda(Ae_k)=\lambda(e_k)=e_1=(1,0,0,\ldots)$. Using the definition of $\lambda$ and the invertibility of $A$, we see the existence of $\varepsilon_k\in \{1,-1\}$ and a permutation  $\sigma$ of $\mathbb{N}$ such that $Ae_k = \varepsilon_ke_{\sigma(k)}$ for all $k$. 
It follows that for any $x=\sum_{k=1}^{\infty}x_ke_k$, 
\[ Ax=\sum_{k=1}^{\infty}x_k\varepsilon_ke_{\sigma(k)}. \]
It is easy to see that every transformation of the above form is an automorphism of the FTvN system $\ell_2(\R)$. We make two important observations: First, the shift operator $S:(x_1, x_2, \ldots) \mapsto (0, x_1, x_2, \ldots)$ leaves every $\lambda$-orbit invariant, but is not invertible; hence it is not an automorphism. Second, if $A$ is an automorphism and $x\in \ell_2(\R)$ has all nonzero entries, then so does $Ax$. This latter observation shows that the elements $\big( 1, \frac{1}{2}, \frac{1}{3}, \ldots \big)$ and $\big( 0, 1, \frac{1}{2}, \frac{1}{3}, \ldots \big)$ are in the same $\lambda$-orbit, but no automorphism can map one to the other. This motivates us to introduce the definition of orbit-transitive system.

\gap

Given an inner product space $\V$, consider a (closed) subgroup $\G$ of the orthogonal group ${\cal O}(\V)$. We say that $\G$ acts {\it transitively} on a set $E\subseteq \V$ if for all $x, y\in E$, there is an $A\in \G$ such that $Ax=y$. For example, the group $\G$ acts transitively on any $\G$-orbit given by
\[ \mathrm{Orb}_{\G}(x) := \{Ax: A\in \G\}. \]

\begin{definition}
	A FTvN system $(\V,\W,\lambda)$ is said to be an orbit-transitive system if the group $\Aut(\V, \W, \lambda)$ is transitive on every $\lambda$-orbit in $\V$, that is, in $\V$, 
	\[ \big[ x, y \in \V ,\, \ly = \lx \big] \Rightarrow y = Ax \,\, \text{for some} \,\, A \in \Aut(\V, \W, \lambda). \]
\end{definition}

Our first result is a characterization of normal decomposition systems among FTvN systems. (We note an attempt made by Orlitzky in \cite{orlitzky-note}.)

\begin{theorem}
	A FTvN system $(\V,\W,\lambda)$ comes from a NDS if and only if it is orbit-transitive, $\W\subseteq \V$, and $\lambda^2=\lambda$.
\end{theorem}

\begin{proof}
	Suppose the FTvN system $(\V,\W,\lambda)$ arises from a NDS $(\V,\G,\gamma)$ so that $\lambda=\gamma$ and $\W:= \spn(\gamma(\V))$. From the properties of NDS, we have $\W\subseteq \V$ and $\lambda^2=\lambda$. We now show that $(\V,\W,\lambda)$ is orbit-transitive. To simplify the notation, let $\A := \Aut(\V,\W,\lambda)$. As observed above, $\G\subseteq \A$. In this setting, 
	\[ [x] = \{y \in \V : \ly = \lx\} = \{y \in \V : \gamma(y)=\gamma(x)\}. \]
	Now, if $u,v\in [x]$, then, $\gamma(u)=\gamma(v)$. Since we are working in a NDS, there exist $A,B\in \G$ such that $u = A\gamma(u)$ and $v = B\gamma(v)$. It follows that $(BA^{-1})(u) = v$ where $BA^{-1}\in \G$. As $\G \subseteq \A$, we see that $C := BA^{-1}\in \A$ with $Cu=v$. Thus, $(\V,\W,\lambda)$ is orbit-transitive. We now prove the converse. Suppose $(\V,\W,\lambda)$ is a FTvN system satisfying the given conditions. We claim that $(\V,\A,\lambda)$ is a NDS by verifying the three conditions in its definition (see the Appendix). Clearly, $\lambda(Ax) = \lx$ for all $x\in \V$ and $A\in \A$. For each $x\in \V$, $\lambda(\lx)=\lx$ and so, $\lx\in [x]$. By orbit-transitivity, there exists $A\in \A$ such that $A\lx=x$. Finally, $\ip{x}{y} \leq \ip{\lx}{\ly}$ for all $x,y \in \V$ from the definition of a FTvN system.
\end{proof}

The FTvN  system $\big (l_2(\R),l_2(\R),\lambda\big )$ is not orbit-transitive, see Example \ref{sequence space} given above. So, this system cannot come from a NDS.

\gap

Now consider an  essentially-simple Euclidean Jordan algebra
$\V$, that is,  $\V$ is either a simple or $\Rn$. In this case, we know that $\V$ can be regarded as a NDS. By the above theorem, the FTvN system $(\V,\Rn,\lambda)$ is orbit-transitive. How about general Euclidean Jordan algebras?
This has been answered recently by Orlitzky \cite{orlitzky-proscribed}. We state his result using our terminology above.

\begin{theorem} (\cite{orlitzky-proscribed}, Theorem 4)
	Suppose $\V$ is a Euclidean Jordan algebra. Then, the following are equivalent:
	\begin{itemize}
		\item [$(i)$] $\V$ is essentially-simple, that is, $\V$ is simple or $\Rn$.
		\item [$(ii)$] Every Jordan frame in $\V$ can be mapped onto any other by an automorphism in $\A := \Aut(\V,\W,\lambda)$.
		\item [$(iii)$] The FTvN system $(\V,\Rn,\lambda)$ is orbit-transitive.
	\end{itemize}
\end{theorem}

In the above result, the equivalence of $(i)$ and $(ii)$ comes from a result of Gowda and Jeong \cite{gowda-jeong}. The implication $(ii) \Rightarrow (iii)$ is known and easy to verify. Orlitzky proves the implication $(iii) \Rightarrow (ii)$ by considering two elements
$x:=1e_1+2e_2+\cdots+ne_n$ and $y=1f_1+2f_2+\cdots+nf_n$ where $\{e_1, e_2, \ldots, e_n\}$ and $\{f_1,f_2,\ldots, f_n\}$ are two arbitrary Jordan frames; since $x$ and $y$ are in the same $\lambda$-orbit, there is an automorphism $A\in \A$ with $Ax=y$. Then, by the uniqueness of spectral decomposition when the eigenvalues are distinct, the Jordan frame $\{e_1, e_2, \ldots, e_n\}$ can be mapped onto $\{f_1,f_2,\ldots, f_n\}$ by $A$.

\gap

\textbf{Remark.} In the setting of a FTvN, we have two types of orbits, namely, the $\lambda$-orbits and $\A$-orbits, with $\A$-orbit being a subset of the $\lambda$-orbit. Recall that a set $E\in \V$ is said to be spectral if $x\in E \Rightarrow [x]\subseteq E$. We define a set $E$ to be {\it weakly spectral} if $x\in E \Rightarrow Orb_{\A}(x)\subseteq E$, that is, $A(E) \subseteq E$ for all $A\in \A$. Clearly, every spectral set is weakly spectral, but the converse may not be true. It is known \cite{gowda-jeong} that in a Euclidean Jordan algebra, every weakly spectral set is spectral if and only if the algebra is essentially-simple.

%%%%%%%%%%%%%%%%%%%%%%%%%%%%%%%%%%%%%%%%%%%%%%%%%%%%%%%%%%%%%%%%%%%%%%%%%%%%%%%%%%%%%%%%%%%%
\section{Majorization}

In this section, we formulate the definition of majorization in FTvN systems. First we recall some standard examples and notation.

\begin{itemize}
	\item Consider the FTvN system  $(\Rn,\Rn,\lambda)$ of Example \ref{rn}.  Given $u,v\in \Rn$, we say that $u$ is {\it weakly majorized} by $v$ and write $u\underset{w}{\prec} v$ if for all natural numbers $k$, $1\leq k\leq n$,
	\[ \sum_{i=1}^{k} u_i^\downarrow \leq \sum_{i=1}^{k} v_i^\downarrow. \] Additionally, if the equality holds when $k=n$, we say that $u$ is {\it majorized} by $v$ and write $u\prec v$ \cite{marshall-olkin}. \\
	{\it A result of Hardy, Littlewood, and P\'{o}lya says that $u\prec v$ if and only if $u = Dv$, where $D$ is an $n\times n$ doubly stochastic matrix} (meaning that all entries of $D$ are nonnegative and each row/column sum is one). Additionally,
	{\it a result of Birkhoff (see \cite{bhatia}) says that every doubly stochastic matrix $D$ is a convex combination of permutation matrices.
	}
	So, in the setting of $\Rn$,
	\[ u \prec v  \,\, \iff \,\, u\in \conv\,[v], \]
	where $[v] = \{Pv : P \in \Sigma_n\} = \{w\in \Rn:\lambda(w)=\lv\}$ is the $\lambda$-orbit of $v$ in $\Rn$ and $\Sigma_n$ is the set of all $n\times n$ permutation matrices. 
	
	\item Let $\Hn$ denote the set of all $n\times n$ complex Hermitian matrices. For any $A \in \Hn$, we let $\lambda(A)$ denote the eigenvalues of $A$ written in the decreasing order. Then, majorization between two complex Hermitian matrices $A$ and $B$ is defined by:
	\[ A\prec B \text{ in } \Hn \,\, \iff \,\, \lambda(A)\prec \lambda(B) \text{ in }\Rn. \]
	In this setting, it is known, see e.g., \cite{alberti-uhlmann}, Theorem 2.2 or \cite{ando}, Theorem 7.1 that  {\it  $A\prec B$ if and only if $A$ is a convex combination of matrices of the form $UBU^*$, where $U$ is a unitary matrix.} Observing that for a complex Hermitian matrix $X$, $\lambda(X)=\lambda(B)$ if and only if $X=UBU^*$ for some unitary matrix $U$, we could state this result as:
	\[ A\prec B \text{ in } \Hn \,\, \iff \,\, A\in \conv\, [B]. \]
	One further observation: The space $\Hn$ of all $n \times n$ complex Hermitian matrices is a Euclidean Jordan algebra in which every algebra automorphism is of the form $X\mapsto UXU^*$ for some unitary matrix $U$.
	
	\item Consider a normal decomposition system $(\V,\G,\gamma)$, see Example \ref{nds}. In this setting, a (group) majorization is defined by \cite{niezgoda-group majorization}:
	\[ x\prec y \,\, \iff \,\, x \in \conv\,\{Ay:A\in \G\}. \]
	With the observation that $(\V,\W,\gamma)$, where $\W := \spn(\lambda(\V))$, is a FTvN system, see \cite{gowda-ftvn}, we note that $[y]=\{Ay:A\in \G\}$; hence, in the FTvN system $(\V,\W,\gamma)$ we have the definition 
	\[ x\prec y \,\, \iff \,\, x \in \conv\, [y]. \]
	
	\item Consider the setting of Example \ref{hyperbolic}: Let $\V$ be a finite dimensional real vector space and $p$ be a hyperbolic  polynomial of degree $n$ on $\V$ (relative to some $e$). Consider the corresponding $\lambda$ map from $\V$  to $\Rn$. For any $y\in \V$, let $[y]:=\{z\in \V:\lambda(z)=\ly\}$ denote the $\lambda$-orbit of $y$. \\
	As in the previous examples, we can define majorization between elements $x$ and $y$ in $\V$ in two ways: Either by the condition $x\in \conv\,[y]$ (which is a condition in $\V)$ or by $\lx\prec \ly$ (which is a condition in $\Rn)$. We now relate these two concepts. 
	
	Based on the validity of the Lax conjecture \cite{lewis et al}, Gurvits \cite{gurvits} has shown that in the
	canonical setting of $\V=\Rn$ and $p(e) = 1$, for any two elements $x, y\in \Rn$ there exist real symmetric
	$n \times n$ matrices $A$ and $B$ such that
	\[ \lambda(tx + sy) = \lambda(tA + sB), \]
	for all $t, s \in  \R$, where the right-hand side denotes the eigenvalue vector of a symmetric matrix.
	This result, when translated to our general $\V$, implies the Lidskii (majorization) property \cite{bhatia}:
	\[ \lx-\ly\prec \lambda(x-y)\quad (x,y\in \V). \]
	This, in particular, yields:
	\[ \lambda(x+y)\prec \lx+\ly\quad (x,y\in \V). \]
	Now suppose that $x,y\in \V$ with $x\in \conv\,[y]$.
	We can then write $x=\sum_{k=1}^{m} t_ky_k$, a convex combination of elements $y_1,y_2,\ldots, y_m$ in $[y]$.
	It follows that
	\[ \lx = \lambda \bigg( \sum_{k=1}^{m} t_k y_k \bigg) \prec \sum_{k=1}^{m} t_k \lambda(y_k) = \ly. \]
	Thus, we arrive at the implication:
	\begin{equation} \label{implication from gurvits}
		x\in \conv\,[y] \Rightarrow \lx \prec \ly.
	\end{equation}
	We note that no assumptions are placed on the hyperbolic polynomial $p$ (such as completeness or isometric property).

\end{itemize}
Motivated by the above examples, we now introduce the concept of majorization in the setting of FTvN systems.

\begin{definition}\label{defn of majorization}
	Let $(\V,\W,\lambda)$ be a FTvN system. Given $x,y\in \V$, we say that $x$ is majorized by $y$ and write $x\prec y$ if 
	$x\in \conv\, [y]$.
\end{definition}

\noindent{\bf Remark.} While we have formulated the definition of majorization in the setting of FTvN systems, one can weaken the definition to include broader systems. With $\V$ and $\W$ denoting real vector spaces and $\lambda:\V \to \W$ denoting a (general) map, one can define majorization $x\prec y$ in $\V$ by the condition $x\in \conv\, [y]$, where $[y]=\{z\in \V: \lambda(z)=\ly\}$. This generalized formulation may be beneficial in some settings (see our hyperbolic example above). Meaningful results, perhaps, are obtained only when $\lambda$ satisfies appropriate conditions.

\gap

Suppose we modify  Definition \ref{defn of majorization} and introduce $*$-majorization: $x \starprec y$ if $x \in \overline{\conv\,[y]}$. Clearly, $x \prec y \Rightarrow x \starprec y$ and the reverse implication holds  when $\V$ is finite dimensional. The following example shows that the two concepts can be different in the infinite dimensional setting. Beyond the example, we will not pursue 
this concept here.

\begin{example}
	Consider the sequence space $\ell_2(\R)$; for any $k \in \mathbb{N}$, let $e^{(k)}$ be the vector with 
	one in the $k$th slot and zeros elsewhere. Take $y = e^{(1)}$. Then, $[y]=\{e^{(1)},e^{(2)},\ldots\}$. Let, for each $n\in \mathbb{N}$,
	\[ x^{(n)}: = \big( \underbrace{\tfrac{1}{n},\, \ldots,\, \tfrac{1}{n}}_{\text{$n$-times}},\, 0,\, 0,\, \ldots \big). \]
	Then, $x^{(n)} \in \conv\,[y]$ for all $n$ and $x_n \to 0$ in $\ell_2(\R)$. We see that $0 \in \overline{\conv\,[y]}$ while $0 \notin \conv\,[y]$. 
\end{example}

\begin{proposition}\label{basic majorization prop} 
	In a FTvN system $(\V,\W,\lambda)$,
	\begin{itemize}
		\item [$(a)$] $x\prec y$ implies $\ip{c}{x} \leq \ip{\lc}{\ly}$ for all $c\in \V$. The reverse implication holds if $\V$ is finite dimensional.
		\item [$(b)$] $x\prec y$ implies $\ip{\lc}{\lx} \leq \ip{\lc}{\ly}$ for all $c \in \V$. The reverse implication  holds if $\V$ is finite dimensional.
		\item [$(c)$] $x\prec y$ and $y\prec x$ if and only if $[x]=[y]$.
	\end{itemize}
\end{proposition}

\begin{proof}
	$(a)$ Suppose $x\prec y$ so that $x=\sum_{i=1}^{k}t_iy_i$, where $t_i$s are nonnegative numbers summing to one and $y_i\in [y]$ for all $i$. Then, for any $c\in \V$,
	\[ \ip{c}{x} = \sum_{i=1}^{k}t_i \ip{c}{y_i} \leq \sum_{i=1}^{k}t_i \ip{\lc}{\lambda(y_i)} = \sum_{i=1}^{k}t_i \ip{\lc}{\ly} = \ip{\lc}{\ly}, \]
	where we have used condition $(A2)$ in Definition \ref{ftvn} and the  assumption that $y_i\in [y]$.
	Thus we have the implication in $(a)$.
	
	Now suppose that $\V$ is finite dimensional and the inequality $\ip{c}{x} \leq \ip{\lc}{\ly}$ holds for all $c \in \V$. Suppose, if possible, $x\not \in \conv\,[y]$. As $\V$ is finite dimensional, $[y]$ is compact; hence  $\conv\,[y]$ is compact and convex. By the strong separation theorem, there exist $c\in \V$ and an $\alpha\in \R$ such that
	\[ \ip{c}{x} > \alpha \geq \ip{c}{u} \,\, \text{for all} \,\, u \in [y]. \]
	Then, by  (\ref{intro ftvn}), 
	\[ \ip{c}{x} > \alpha \geq \max_{u\in [y]} \ip{c}{u} = \ip{\lc}{\ly}. \] 
	As this contradicts our assumption, we have $x\in \conv\, [y]$, justifying  the reverse implication.
	
	\smallgap
	
	$(b)$ Suppose $x\prec y$.
	We fix $c$ and let $z\in [c]$. Then, by $(a)$, $\ip{z}{x} \leq \ip{\lambda(z)}{\ly} = \ip{\lc}{\ly}$. Maximizing over $z$ in $[c]$ and applying (\ref{intro ftvn}), we see that $\ip{\lc}{\lx} \leq \ip{\lc}{\ly}$. To see the reverse implication, assume that for all $c \in \V$, $\ip{\lc}{\lx} \leq \ip{\lc}{\ly}$. Then, $\ip{c}{x} \leq \ip{\lc}{\ly}$. Since $\V$ is assumed to be finite dimensional, we can apply $(a)$ to get $x\prec y$. 
	
	\smallgap
	
	$(c)$ Let $x \prec y$ and $y\prec x$ so that we have convex combinations $x = \sum_{i=1}^{k} t_iy_i$ and $y = \sum_{j=1}^{l} s_jx_j$, where $t_i$ and $s_j$ are positive scalars, $\sum_{i=1}^{k} t_i = 1$, and $\sum_{j=1}^{l} s_j = 1$, $y_i \in [y]$ and $x_j \in [x]$ for all $i,j$.
	Since $y_i\in [y] \Rightarrow \lambda(y_i)=\ly \Rightarrow \norm{y_i} = \norm{y}$, we see that $\norm{x} \leq \sum_{i=1}^{k} t_i \norm{y_i} = \norm{y}$; similarly, $\norm{y} \leq \norm{x}$. Thus, 
	$x$ is a convex combination of $y_i$s, where $\norm{y_i} = \norm{x}$ for all $i$. By strict convexity of the (inner product) norm, we see that $x=y_i$ for all $i$, showing that $x\in [y]$. Likewise $y\in [x]$. Thus, $[x]=[y]$.
	This proves the `only if' part. The `if' part is obvious.
\end{proof}

%%%%%%%%%%%%%%%%%%%%%%%%%%%%%%%%%%%%%%%%
\section{Doubly stochastic transformations}
In this section, we formulate the definition of a doubly stochastic transformation on a FTvN system. 

\gap

First, we recall some standard definitions. A {\it doubly stochastic matrix} is a real $n\times n$ nonnegative matrix, each of whose rows and columns sums to one. A linear transformation $D$ on $\Hn$ is {\it doubly stochastic} \cite{ando} if it positive, unital, and  trace-preserving, that is, it keeps the semidefinite cone invariant, $D(I)=I$ and $D^*(I)=I$, where $I$ is the identity matrix and $D^*$ denotes the adjoint of $D$. More generally, a linear transformation $D$ on a  Euclidean Jordan algebra $\V$ is {\it doubly stochastic} \cite{gowda-doubly-stochastic} if it keeps the corresponding symmetric cone invariant, $D(e)=e$ and $D^*(e)=e$, where
the symmetric cone $\V_+$ consists of all elements of $\V$ with nonnegative eigenvalues, $e$ is the  unit element, and $D^*$ denotes the adjoint. 

\gap

Now, deviating from the use of specific cones and/or unit elements, we introduce  the definition of a doubly stochastic transformation in the setting of a FTvN system and show that this definition is equivalent to the standard ones in familiar settings.

\begin{definition}
	Consider a FTvN system $(\V,\W,\lambda)$. A linear transformation $D:\V \to \V$ is said to be doubly stochastic if $Dx \prec x$ for all $x\in \V$, that is,
	$Dx\in \conv [x]$ for all $x\in \V$. 
\end{definition}

In the proof of Item $(c)$ of Proposition \ref{basic majorization prop}, we observed that $x \prec y \Rightarrow \norm{x} \leq \norm{y}$. This shows that every doubly stochastic transformation is continuous. 

We recall the definition of adjoint $D^*$ of $D$: For all $x,y\in \V$, $\ip{D^*x}{y} = \ip{x}{Dy}$; note that $D^*$ exists, for example, when $\V$ is a Hilbert space. 

\begin{proposition}
	Suppose $D$ is a doubly stochastic transformation on 
	$(\V,\W,\lambda)$ with adjoint $D^*$. Then, the following statements hold:
	\begin{itemize}
		\item [$(a)$] For every convex spectral set $K$, $D(K)\subseteq K$. 
		\item [$(b)$] For any $u$ in (the center) $C$, $Du=u$ and $D^*u=u$. In particular, if $e$ is a unit element of $(\V, \W, \lambda)$, then $De=e$ and $D^*e=e$.
		\item [$(c)$] If $\V$ is finite dimensional, then $D^*x\prec x$ for all $x\in \V$.
	\end{itemize}
\end{proposition}

\begin{proof}
	$(a)$ Suppose $K$ is a  spectral set that is also convex. For any $x\in K$, we have $[x]\subseteq K$ and hence, $\conv\,[x]\subseteq K$. Then, $Dx\in \conv\,[x]$ implies that $Dx\in K$.
	
	\smallgap
	
	$(b)$ We have observed previously (see Proposition \ref{e-orbit prop}) that $[u]=\{u\}$ for all $u \in C$. So, when $D$ is doubly stochastic, we have
	$Du\prec u$, that is, $Du\in \conv\,[u]=\{u\}$. Thus, $Du=u$. Suppose, if possible, $D^*u\neq u$. Then (working in a finite dimensional subspace), we can find a $d\in \V$ such that $\ip{d}{D^*u} > \ip{d}{u}$. Since $u$ commutes with all elements of $\V$, this leads to $\ip{\lambda(Dd)}{\lu} = \ip{Dd}{u} = \ip{d}{D^*u} > \ip{d}{u}$. However, from Proposition \ref{basic majorization prop}$(b)$, 
	\[ Dd \prec d \Rightarrow \ip{\lambda(Dd)}{\lu} \leq \ip{\lambda(d)}{\lu}=\ip{d}{u}. \] 
	We reach a contradiction. Hence, $D^*u=u$.
	The additional statement follows as $e\in C$.
	
	\smallgap
	
	$(c)$ Let $\V$ be finite dimensional. Suppose, if possible, there exists $c\in \V$ with $D^*c\not\prec c$, that is, $D^*c\not\in \conv [c]$. 
	Since $\V$ is finite dimensional, $\conv\,[c]$ is compact and convex; so, by the strong separation theorem, we can find  $d\in \V$ and $\alpha\in \R$ such that 
	\[ \ip{d}{D^*c} > \alpha \geq \ip{d}{u}, \]
	for all $u\in [c]$. Taking the maximum over $u$ in $[c]$ and using  (\ref{intro ftvn}), we see that 
	\[ \ip{Dd}{c} > \alpha \geq \ip{\lambda(d)}{\lc}. \]
	However, $Dd \prec d$ implies, via Proposition \ref{basic majorization prop}$(a)$ that $\ip{Dd}{c} \leq \ip{\lambda(d)}{\lc}$. We reach a contradiction. Hence, $D^*x \prec x$ for all $x\in \V$.
\end{proof}

\noindent{\bf Remark.} We note that when $\V$ is finite dimensional, Item $(a)$  in the above proposition is 
equivalent to $D$ being doubly stochastic. 
This is because, when $\V$ is finite dimensional, for any $x\in \V$, $K:=\conv\,[x]$ is a convex spectral set (see Proposition \ref{spectral invariance prop}(c)); so $D(K)\subseteq K$ implies that $Dx\in \conv\,[x]$, or equivalently, $Dx\prec x$.

\gap

The following result shows that in the setting of a Euclidean Jordan algebras (hence in $\Hn$ and $\Rn$), our  definition coincides with the standard ones.

\begin{corollary}
	Consider the FTvN system $(\V,\Rn,\lambda)$ corresponding to the Euclidean Jordan algebra $\V$, see Example \ref{eja}. Let $e$ be the unit element of $\V$ and $D:\V \to \V$ be linear. Then the following statements are equivalent:
	\begin{itemize}
		\item [$(i)$] $D$ is doubly stochastic on the FTvN system $(\V,\Rn,\lambda)$, i.e., $Dx\in \conv\, [x]$ for all $x\in \V$. 
		\item [$(ii)$]  $D$ is doubly stochastic on the Euclidean Jordan algebra $\V$, i.e., $D(\V_+)\subseteq \V_+$, $De=e$ and $D^*e=e$.
		\item [$(iii)$] $\lambda(Dx)\prec \lx$ in $\Rn$ for all $x\in \V$.
	\end{itemize}
\end{corollary}

\begin{proof}
	$(i) \Rightarrow (ii)$: This follows immediately from Items $(a)$ and $(b)$ of the previous result, as  the symmetric cone $\V_+$ is a convex spectral cone, the unit element $e$ in the algebra is a  unit element in the corresponding FTvN system.
	
	\smallgap
	
	$(ii) \Rightarrow (iii)$: This has been proved in Theorems 5 and 6 in \cite{gowda-doubly-stochastic}. Here, for completeness, we repeat the arguments. Fix $x\in \V$ and let $y:=Dx$. We write the spectral decompositions of $x$ and $y$: $x = \lambda_1(x)e_1 + \lambda_2(x)e_2 + \cdots + \lambda_n(x)e_n$ and $y = \lambda_1(y)f_1 + \lambda_2(y)f_2 + \cdots + \lambda_n(y)f_n$, where $\{e_1, e_2, \ldots, e_n\}$ and $\{f_1, f_2, \ldots, f_n\}$ are Jordan frames in $\V$. Then for all $i$,
	\[ \ly_i = \ip{y}{f_i} =\ip{Dx}{f_i} = \sum_{j=1}^{n} \ip{De_j}{f_i} \lambda_j(x). \]
	From the properties imposed on $D$, the matrix $M = [m_{ij}]$, where $m_{ij} = \ip{De_j}{f_i}$, is doubly stochastic. Since $\ly = M\lx$, we see that $\ly \prec \lx$ in $\Rn$. This gives Item $(iii)$.
	
	\smallgap
	
	$(iii) \Rightarrow (i)$: Assume $(iii)$. We consider the FTvN system $(\Rn, \Rn, \mu)$, where $\mu(p) = p^\downarrow$. Now, for any $x \in \V$,  $\lambda(Dx) \prec \lx$ in $\Rn$. From Proposition \ref{basic majorization prop}$(b)$ applied to the FTvN system $(\Rn, \Rn, \mu)$, we get, for any $c\in \V$,
	\[ \ip{\mu(\lc)}{\mu(\lambda(Dx))} \leq \ip{\mu(\lc)}{\mu(\lx)}. \]
	However, we have $\mu \circ \lambda = \lambda$. Hence, 
	\[ \ip{\lc}{\lambda(Dx)} \leq \ip{\lc}{\lx}. \]
	As $\V$ is finite dimensional, from Proposition \ref{basic majorization prop}$(b)$ applied to the FTvN system $(\V, \Rn, \lambda)$, we see that $Dx \prec x$ in $(\V, \Rn, \lambda)$. Thus we have $(i)$.
\end{proof}

We make an important observation:  In the setting of the FTvN system coming from a  Euclidean Jordan algebra, apart from $De=e$ and $D^*(e)=e$, we require $D$ to keep only one  convex
spectral set, namely  $\V_+$, invariant. Do we have such a result for the system coming from a hyperbolic polynomial? 

\gap

\noindent{\bf Problem.} Consider the FTvN system of Example \ref{hyperbolic}. Then, the {\it hyperbolicity cone} is, by definition, 
\[ \V_+:=  \{ x \in \V : \lx \geq 0 \}. \]
Suppose $D:\V \to \V$ is linear with $D(\V_+)\subseteq \V_+$ and $D(e)=e=D^*(e)$. Does it follow that $D$ is doubly stochastic, that is, $Dx\in \conv\,[x]$ for all $x\in \V$?

\begin{proposition} In a FTvN system $(\V,\W,\lambda)$, the following hold:
	\begin{itemize}
		\item [$(a)$] Let $A:\V \to \V$ be linear. Then, $A \in \Aut(\V,\W,\lambda)$ if and only if $A$ is invertible with $A$ and $A^{-1}$ both doubly stochastic.
		\item [$(b)$] If $\V$ is finite dimensional, every $D$ in the convex hull of $\Aut(\V, \W, \lambda)$ is doubly stochastic.
	\end{itemize}
\end{proposition}

\begin{proof}
	$(a)$ Suppose $A \in \Aut(\V,\W,\lambda)$. Then for all $x\in \V$, we have $\lambda(Ax) = \lx = \lambda(A^{-1}x)$; hence $Ax, A^{-1}x \in [x]$. So, $Ax \prec x$ and $A^{-1}x \prec x$ for all $x \in \V$. Thus, $A$ and $A^{-1}$ are doubly stochastic. Conversely, suppose $A$ is invertible with $Ax \prec x$ and $A^{-1}x \prec x$ for all $x \in \V$. As $A^{-1}x \prec x$ for all $x \in \V$ is equivalent to $y \prec Ay$ for all $y \in \V$, by Proposition \ref{basic majorization prop}$(c)$, we see that $[Ay]=[y]$ for all $y \in \V$. Thus, we have $A \in \Aut(\V,\W,\lambda)$.
	
	\smallgap
	
	$(b)$ Now suppose $\V$ is finite dimensional and $D=\sum_{i=1}^{k} t_i A_i$ is a convex combination of $A_1,A_2,\ldots, A_k$ with $A_i \in \Aut(\V,\W,\lambda)$ for all $i$. To show that $Dx\prec x$ for all $x$, it is enough to show that $\ip{\lc}{\lambda(Dx)} \leq \ip{\lc}{\lx}$ for all $c\in \V$ and quote Proposition \ref{basic majorization prop}$(b)$. To this end, we fix $c$ and observe, due to sublinearity, 
	\[ \Big\langle \lc, \lambda \Big( \sum_{i=1}^{k} A_ix \Big) \Big\rangle \leq \Big\langle \lc, \sum_{i=1}^{k} t_i \lambda(A_ix) \Big\rangle = \ip{\lc}{\lx}, \]
	where we have used the fact that $\lambda(A_ix) = \lx$ for all $i$. Thus, $Dx\prec x$ for all $x$, so $D$ is doubly stochastic.
\end{proof}

\begin{corollary}
	Suppose $\V$ is finite dimensional and $\Aut(\V,\W,\lambda)$ is orbit-transitive. If $x\prec y$, then there exists a doubly stochastic transformation $D$ (which is a convex combination of automorphisms on $(\V,\W,\lambda)$) such that $x=Dy$. 
\end{corollary}

\begin{proof}
	If $x\prec y$, then $x\in \conv\,[y]$ and so $x$ is a convex combination of some $y_i$ in $[y]$, $i=1,2,\ldots, k$. As $(\V,\W,\lambda)$ is orbit-transitive, there exists $A_i\in \Aut(\V,\W,\lambda)$ such that $y_i = A_i y$ for all $i$. Thus, $x = \big( \sum_{i=1}^{k} t_i A_i \big) y$, where $t_i$s are nonnegative, adding up to one. Since $\V$ is finite dimensional, we can apply the previous result to conclude that $D: = \sum_{i=1}^{k} t_i A_i$ is doubly stochastic.
\end{proof}

\gap

\noindent{\bf Remark.} Suppose $(\V,\G,\gamma)$ is a finite dimensional NDS. If $x\prec y$ in this system, we can find $D$ which is a convex combination of elements $A_1,A_2,\ldots, A_k$ in $\G$ such that $x=Dy$. 
(This is because, in the above proof, $y_i=A_iy$ for some $A_i\in \G$.)

%%%%%%%%%%%%%%%%%%%%%%%%%%%%%%%%%%%%%%%%%%%%%%%%%%%%%%%%%%%%%%%%
\section{Reduced systems}

In certain settings, the properties of a FTvN system $(\V,\W,\lambda)$ are closely related to those of a companion system $(\W,\W,\mu)$. The goal of this section is to understand the interplay between the two systems, specifically addressing the implication $x\prec y \Rightarrow \lx\prec \ly$ and Lidskii type inequality $\lambda(x+y)\prec \lx+\ly$. To motivate, we consider the following  examples.

\begin{itemize}
	\item Suppose $\V$ is a Euclidean Jordan algebra of rank $n$ carrying the trace inner product with $\lambda:\V \to \Rn$ denoting the eigenvalue map. Then, as in Example \ref{eja}, $(\V,\Rn,\lambda)$ is a FTvN system. In addition, $(\Rn,\Rn,\mu)$ is a FTvN system, where $\mu(q)=q^\downarrow$. In this setting, $\mu\circ \lambda=\lambda$ and $\mathrm{ran}\,\mu = \mathrm{ran}\,\lambda$. 
	
	\item We consider a FTvN system $(\V,\Rn,\lambda)$ induced by a complete isometric hyperbolic polynomial, see Example \ref{hyperbolic}. With $\mu$ as in the previous example, $(\Rn,\Rn,\mu)$ is a FTvN system, and moreover $\mu\circ \lambda=\lambda$. While the condition $\mathrm{ran}\,\mu = \mathrm{ran}\,\lambda$ may not always hold, such a condition was imposed in getting an important subgradient formula (\cite{bauschke et al}, Theorem 5.5). 
	
	\item Consider a finite dimensional NDS (equivalently, an Eaton triple) $(\V,\G,\gamma)$. As noted in Example \ref{nds}, $(\V,\W,\lambda)$ is a FTvN system, where $\W = \mathrm{span}(\gamma(\V))$ and $\lambda = \gamma$. Now, let  $\h:=\{A\in \G: A(\W)=\W\}$ and $\mu=\gamma|_{\W}$. Then,   under appropriate conditions (see \cite{niezgoda-group majorization}, Theorem 3.2) $(\W,\h,\mu)$ becomes an Eaton triple, or equivalently, $(\W,\W,\mu)$ becomes a NDS. In this setting, we see that $\mu\circ \lambda=\lambda$ and $\mathrm{ran}\,\mu\subseteq \mathrm{ran}\,\lambda$. 
\end{itemize}

Motivated by these examples, we introduce the following.

\begin{definition} (Reduced system)
	Let $(\V,\W,\lambda)$ be a FTvN system. Suppose $(\W,\W,\mu)$ is a FTvN system such that 
	\begin{itemize}
		\item [$(C1)$] $\mu\circ \lambda=\lambda$, and
		\item [$(C2)$] $\mathrm{ran}\,\mu \subseteq \mathrm{ran}\,\lambda$.
	\end{itemize}
	Then,  $(\W,\W,\mu)$ is called a reduced system of $(\V,\W,\lambda)$.
\end{definition}

\begin{proposition} 
	Suppose $(\W,\W,\mu)$ is a reduced system of $(\V,\W,\lambda)$. Let $C_{\V}$ and $C_{\W}$ denote the centers of $(\V,\W,\lambda)$ and $(\W,\W,\mu)$, respectively. Then, the following statements hold:
	\begin{itemize}
		\item [$(a)$] $\mathrm{ran}\,\mu = \mathrm{ran}\,\lambda$ and $\mu^2=\mu$.
		\item [$(b)$] $\lambda(C_{\V}) = \mu(C_{\W})=C_{\W}$. 
		\item [$(c)$] $\mathrm{dim}\,(C_{\V}) = \mathrm{dim}\,\lambda(C_{\V}) = \mathrm{dim}\,\mu(C_{\W}) = \mathrm{dim}\,(C_{\W})$. 
		\item [$(d)$] $e$ is a unit in $\V$ if and only if $\lambda(e)$ is a unit in $\W$.
	\end{itemize}
\end{proposition}

\begin{proof}
	$(a)$ The first equality can be seen by combining $(C1)$ and $(C2)$. For the second one, let 
	$w\in \W$. Then, by $(C2)$, $\mu(w)=\lx$ for some $x\in \V$. Using $(C1)$, $\mu(\mu(w)) = \mu(\lx) = \lx = \mu(w)$.
	Hence, $\mu^2=\mu$.
	
	\smallgap
	
	$(b)$ We combine Corollary \ref{lineality result} and $(a)$ above to see
	\[ \lambda(C_{\V}) = \lambda(\V) \cap - \lambda(\V) = \mathrm{ran}\, \lambda \cap - \mathrm{ran}\, \lambda = \mathrm{ran}\, \mu \cap - \mathrm{ran}\, \mu = \mu(C_{\W}). \]
	Now, from Proposition \ref{e-orbit prop} (applied to $(\W,\W,\mu)$), the $\mu$-orbit of any element of $C_{\W}$ is a singleton. However, for any $w\in \W$, we have $\mu(\mu(w))=\mu(w)$; hence, $\mu(w)$ belongs to the $\mu$-orbit of $w$. Thus, when $w\in C_{\W}$, we get $\mu(w)=w$. We see that $\mu(C_{\W})=C_{\W}$.
	Hence, we  have proved $(ii)$.
		
	\smallgap
	
	$(c)$ We know that $C_{\V}$ is a subspace in $\V$ and $\lambda$ is linear on it. Since $\lambda$ is norm-preserving, it is also one-to-one on $C_{\V}$. Hence $C_{\V}$ and $\lambda(C_{\V})$ have the same dimension. As  $\lambda(\V) = C_{\W}$, we see that $C_{\V}$ and $C_{\W}$ also have the same dimension.
	
	\smallgap
	
	$(d)$ By $(c)$, $C_{\V}$ is one-dimensional if and only if $C_{\W}$ is one-dimensional. Thus, $(\V,\W,\lambda)$ has a unit if and only if $(\W,\W,\mu)$ has a unit. Now suppose that $e$ is a unit in $\V$, that is, $e$ is nonzero and  $C_{\V}=\R\,e$. By $(b)$ and the linearity of $\lambda$ on $C_{\V}$, we have  $C_{\W}=\R\,\lambda(e)$. Since $\lambda$ is also norm-preserving, we have $\lambda(e)\neq 0$. Thus, $\lambda(e)$ is a unit in $\W$.
\end{proof}

Suppose $(\W,\W,\mu)$ is a reduced system of $(\V,\W,\lambda)$. Then, by above, $\mu^2=\mu$. On the other hand, if  $(\W,\W,\mu)$ is a FTvN system with $\mu^2=\mu$, then,  by taking $\V=\W$ and $\lambda=\mu$, we see that $(\W,\W,\mu)$ is a reduced system of $(\V,\W,\lambda)$. Thus, a system $(\W,\W,\mu)$ is the reduced system of some FTvN system if and only if $\mu^2=\mu$.
So, every NDS is a reduced FTvN system (of itself).
By an earlier result, a FTvN system $(\W,\W,\mu)$ with $\mu^2=\mu$ comes from an NDS if and only if the system is orbit-transitive. 

\gap

\noindent{\bf Problem.} {\it Characterize FTvN system $(\W,\W,\mu)$, where $\W$ is finite dimensional and $\mu^2=\mu$.}

\begin{theorem}\label{Lidskii type theorem}
	Suppose $(\W,\W,\mu)$ is a reduced system of $(\V,\W,\lambda)$ with $\W$ finite dimensional. Let $F:= \mathrm{ran}\,\lambda$ and $F^*$ denote the dual of the cone $F$ in $\W$. Then,
	the following statements hold:
	\begin{itemize}
		\item [$(a)$] If $u,v\in F$ and $u-v \in F^*$, then $v \prec u$ in $\W$.
		\item [$(b)$] For all $x_1,x_2,\ldots, x_k\in \V$, $\lambda(x_1+x_2+\cdots+x_k)\prec \lambda(x_1)+\lambda(x_2)+\cdots+\lambda(x_k)$ in $\W$.
		\item [$(c)$] $x\prec y$ in $\V$ implies $\lx \prec \ly$ in $\W$. The converse holds if $\V$ is finite dimensional.
	\end{itemize}
\end{theorem}

\begin{proof}
	$(a)$ Suppose $u,v\in F$, $u-v\in F^*$, and $v \not\in \conv\,[u]$ in $\W$. As $\W$ is finite dimensional, $\conv \,[u]$ is compact and convex; hence,
	by the strong separation theorem, there exist $d\in \W$ and  $\alpha\in \R$ such that
	\[ \ip{d}{v} > \alpha \geq \max_{w\in [u]} \ip{d}{w}. \]
	Now, in the system $(\W, \W, \mu)$ we have
	\[ \ip{\mu(d)}{\mu(v)} \geq \ip{d}{v} \,\, \text{and} \,\, \max_{w\in [u]} \ip{d}{w} = \ip{\mu(d)}{\mu(u)}. \]
	Hence,
	\[ \ip{\mu(d)}{\mu(v)} > \ip{\mu(d)}{\mu(u)}. \]
	However, from $(C1)$, $u, v \in F \Rightarrow \mu(u)=u, \,\mu(v)=v$ and, from $(C2)$, $r:=\mu(d)\in F$. The above inequality implies that 
	$\ip{r}{v} > \ip{r}{u}$, that is, $\ip{r}{u-v} < 0$, contradicting the assumption $u-v\in F^*$. We thus have Item $(a)$.
	
	\smallgap
	
	$(b)$ From Theorem \ref{basic theorem}, for all $c\in \V$, 
	\[ \Big\langle \lc, \lambda \Big( \sum_{i=1}^{k}x_i \Big) \Big\rangle \leq \sum_{i=1}^{k} \ip{\lc}{\lambda(x_i)}. \]
	Let $u:=\sum_{i=1}^{k}\lambda(x_i)$, $v:=\lambda\big (\sum_{i=1}^{k}x_i\big )$, and $r:=\lc$. Then, $\ip{r}{v} \leq \ip{r}{u}$. As $u, v \in F$ and $r=\lc$ is arbitrary in $F$ (=$\lambda(\V)$), we see that $u-v\in F^*$. From Item $(a)$, $v\prec u$, that is,
	\[ \lambda(x_1 + x_2 + \cdots + x_k) \prec \lambda(x_1) + \lambda(x_2) + \cdots + \lambda(x_k). \]
	
	\smallgap
	
	$(c$) Suppose $x\prec y$ in $\V$ so that $x = \sum_{i=1}^{k}t_iy_i$, where $t_i$s are nonnegative with sum one and $y_i\in [y]$ for all $i$. By using the (positive) homogeneity of $\lambda$ and Item $(b)$, we see that 
	\[ \lx\prec \sum_{i=1}^{k}t_i\lambda(y_i)=\ly, \]
	as $\lambda(y_i)=\ly$ for all $i$.
	
	Now suppose that $\V$ is finite dimensional, $\lx\prec \ly$ in $\W$, but $x\not\prec y$ in $\V$, that is, $x\not \in \conv\,[y]$. As in the proof of Item $(a)$ in Proposition \ref{basic majorization prop}, there exists a $c\in \V$ such that $\ip{c}{x} > \ip{\lc}{\ly}$. On the other hand, by applying Item $(a)$ in Prop \ref{basic majorization prop} in the FTvN system $(\W,\W,\mu)$, we see that 
	\[ \lx\prec \ly \Rightarrow \ip{\lc}{\lx} \leq \ip{\mu(\lc)}{\mu(\ly)}. \] 
	With $\mu(\lc)=\lc$ and $\mu(\ly)=\ly$ (coming from condition $(C1)$), we see that $\ip{\lc}{\lx} \leq \ip{\lc}{\ly}$, which contradicts a previous inequality. Thus, $x \prec y$ in $\V$.
\end{proof}

\begin{corollary}
	Suppose $(\W,\W,\mu)$ is a reduced system of $(\V,\W,\lambda)$, where both $\V$ and $\W$ are finite dimensional. Let $x,y\in \V$. Then,
	\[ x\prec y \text{ in } \V \,\, \big( i.e.,\, x\in \conv\,[y] \big) \,\, \iff \,\, \lx \prec \ly \text{ in } \W. \]
	In particular, if the system $(\V,\Rn,\lambda)$ comes from a Euclidean Jordan algebra $\V$, then
	\[ x\prec y \text{ in } \V \,\, \big( i.e.,\, x\in \conv\,[y] \big) \,\, \iff \,\, \lx \prec \ly \text{ in } \Rn. \]
\end{corollary}

\begin{proposition} 
	Consider the FTvN system $(\V,\Rn,\lambda)$ induced from a complete and isometric hyperbolic polynomial (see Example \ref{hyperbolic}). Let $x,y\in \V$. Then,
	\[ x\prec y \text{ in } \V \,\, \big( i.e.,\, x\in \conv\,[y] \big) \,\, \iff \,\, \lx \prec \ly \text{ in } \Rn. \]
\end{proposition}

\noindent{\it Note:} The previous corollary cannot be applied here because the condition $\mathrm{ran}\,\mu \subseteq \mathrm{ran}\,\lambda$ with $\mu(q) = q^\downarrow$ on $\Rn$ may not hold; so, $(\Rn, \Rn, \mu)$  may not be a reduced system of $(\V, \Rn, \lambda)$.

\gap

\begin{proof}
	The implication $x\in \conv\,[y] \Rightarrow \lx \prec \ly$ has already been observed for any hyperbolic polynomial, see \eqref{implication from gurvits}. 
	For the  reverse implication, we mimic the proof given in Item $(c)$ of Theorem  \ref{Lidskii type theorem} with the observation that condition $(C1)$ holds in our setting of $(\V, \Rn, \lambda)$ and $(\Rn, \Rn, \mu)$, where $\mu(q)=q^\downarrow$.
\end{proof}

\noindent{\bf Remark.} A number of interesting results on reduced systems of Eaton triples appear in \cite{tam}.

%%%%%%%%%%%%%%%%%%%%%%%%%%%%%%%%%%%%%%%%%%%%%%%%%%%%%%%%%%%%%%%%%%%%%%%%%
\section{Appendix}

For ease of reference, we now provide the definitions of normal decomposition systems and Eaton triples.

\subsection{Normal decomposition systems}

\begin{definition} \cite{lewis}\label{definition: nds}
	Let $\V$ be a real inner product space, $\G$ be a closed subgroup of the orthogonal group of $\V$, and $\gamma : \V  \to \V$ be a map satisfying the following conditions:
	\begin{itemize}
		\item [$(a)$] $\gamma$ is $\G$-invariant, that is, $\gamma(Ax) = \gamma(x)$ for all $x \in \V$ and $A \in \G$.
		\item [$(b)$] For each $x\in \V$, there exists $A\in \G$ such that $x=A\gamma(x)$.
		\item [$(c)$] For all $x,y\in \V$, we have $\ip{x}{y} \leq \ip{\gamma(x)}{\gamma(y)}$.
	\end{itemize}
	Then, $(\V, \G, \gamma)$ is called a {\it normal decomposition system}.
\end{definition}

Items $(a)$ and $(b)$ in the above definition show that $\gamma^2=\gamma$ and $\norm{\gamma(x)} = \norm{x}$ for all $x$.
We state a few relevant properties.

\begin{proposition} [\cite{lewis}, Proposition 2.3 and Theorem 2.4] \label{prop: lewis}
	Let $(\V,\G,\gamma)$ be a normal decomposition system. Then,
	\begin{itemize}
		\item [$(i)$] For any two elements $x$ and $y$ in $\V$, we have
		\[ \max_{A \in \G}\, \ip{Ax}{y} = \ip{\gamma(x)}{\gamma(y)}. \]
		Also, $\ip{x}{y} = \ip{\gamma(x)}{\gamma(y)}$ holds for two elements $x$ and $y$ if and only if there exists an $A \in \G$ such that $x = A\gamma(x)$ and $y = A\gamma(y)$.
		
		\item [$(ii)$] The range of $\gamma$, denoted by $F$, is a closed convex cone in $\V$.
	\end{itemize}
\end{proposition}

\subsection{Eaton triples}
These were introduced and studied in \cite{eaton-perlman, eaton1, eaton2} from the perspective of majorization techniques in probability. They were also extensively studied in the papers of Tam and Niezgoda, see the references.

\begin{definition} \label{definition: eaton triple}
	Let $\V$ be a finite dimensional real inner product space, $\G$ be a closed subgroup of the orthogonal group of $\V$, and $F$ be a closed convex cone in $\V$ satisfying the following conditions:
	\begin{itemize}
		\item [$(a)$] $Orb(x)\cap F \neq \emptyset$ for all $x \in \V$, where $Orb(x) := \{Ax : A \in \G\}$.
		\item [$(b)$] $\ip{x}{Ay} \leq \ip{x}{y}$ for all $x, y \in F$ and $A \in \G$.
	\end{itemize}
	Then, $(\V, \G, F)$ is called an {\it Eaton triple}.
\end{definition}

It has been shown (see \cite{niezgoda-group majorization}, page 14) that in an Eaton triple $(V, \G, F)$, $Orb(x)\cap F$ consists of exactly one element for each $x\in \V$. Defining $\gamma : \V \to \V$ such that $Orb(x) \cap F = \{\gamma(x)\}$, it has been observed that $(\V, \G, \gamma)$ is a normal decomposition system. Also, given a finite dimensional normal decomposition system $(\V,\G, \gamma)$, with $F := \gamma(\V)$, $(\V, \G, F)$ becomes an Eaton triple. Thus, {\it finite dimensional normal decomposition systems are equivalent to Eaton triples} \cite{lewis, lewis3, niezgoda-commutation}.

\subsection{Rearrangement inequality for measurable functions}
The notion of rearrangement of a function, systematically introduced by Hardy and Littlewood, has played a key role in proving inequalities in classical and applied analysis. The definitions and properties in this subsection can be found in \cite{bennett-sharpley}, Chapter 2. 

\gap

Let $(\Omega, \Sigma, \mu)$ denote a $\sigma$-finite measure space. 

\begin{definition}
	Let $f : \Omega \to \R$ be a $\Sigma$-measurable function. 
	\begin{itemize}
		\item The function $\mu_f : [0, \infty) \to [0, \infty]$ defined by
		\[ \mu_f(\alpha) = \mu \big( \{x \in \Omega : \abs{f(x)} > \alpha \} \big) \]
		is called the distribution function of $f$.
		\item The decreasing rearrangement of $f$ is the function $f^* : [0, \infty) \to [0, \infty]$ defined by
		\[ f^*(t) := \inf \{ \alpha \geq 0 : \mu_f(\alpha) \leq t \}, \]
		where we use the convection that $\inf \emptyset = \infty$.
	\end{itemize}
\end{definition}
The next propositions establish some basic properties of the distribution function and decreasing rearrangement.

\begin{proposition}
	The following properties hold:
	\begin{itemize}
		\item[(i)] $f$ and $f^*$ are equimeasurable, that is,
		\[ \mu \big( \{x \in \Omega : \abs{f(x)} > \alpha \} \big) = m \big( \{t > 0 : f^*(t) > \alpha \} \big) \]
		for all $\alpha \geq 0$, where $m$ is the Lebesgue measure.
		
		\item[(ii)] $\norm{f}_{L_{p}(\Omega)} = \norm{f^*}_{L_{p}[0, \infty)}$ for all positive real numbers $p$.
		
		\item[(iii)] The Hardy-Littlewood-P{\'o}lya inequality holds, i.e., 
		\[ \int _{\Omega} \abs{fg} \; d\mu \leq \int_{0}^{\infty} f^{*}(t)g^{*}(t) \; dt. \]
	\end{itemize}

\end{proposition}

\begin{proposition}
	There exists only one right-continuous decreasing function $f^*$ equimeasurable with $f$. Hence, the decreasing rearrangement is unique.
\end{proposition}

\gap

{\bf Example} Consider the measure space $(\Omega, \Sigma, \mu) = (\mathbb{N}, 2^{\mathbb{N}}, \mu)$, where $\mu$ is  the counting measure on $\mathbb{N}$. Then, any $\Sigma$-measurable function $f : \Omega \to \R$ can be realized as a sequence $(x_n)$. The decreasing rearrangement $f^*$ is a function defined on $[0, \infty)$, but  can be interpreted as a sequence $(x_n^*)$, where, for any $n\in \mathbb{N}$, 
\[ x_n^* := f^*(t) \;\; \text{for} \;\; n-1 \leq t < n. \]
Formally, for any $n$,

\[ x_n^* = \inf \{ \alpha \geq 0 : \mu \big( \{k \in \mathbb{N} : \abs{x_k} > \alpha \} \big) \leq n-1 \}. \]

The following can easily be observed:
\begin{itemize}
	\item [$(a)$] If $(x_n)$ has finite number of nonzero entries, say, $k$ of them, then 
	$x^*=(|x_{n_1}|,|x_{n_2}|,\ldots, |x_{n_k}|, 0, 0, \ldots)$, with $|x_{n_1}| \geq |x_{n_2}| \geq \cdots \geq |x_{n_k}|$.
	\item [$(b)$] If $(x_n)$ has infinitely many nonzero entries, then $x^*$ consists of absolute values of these entries arranged in the decreasing order; in particular, every entry of $x^*$  is nonzero.
\end{itemize}
For example, if $x = \Big( 1, 0, \frac{1}{2}, 0, \frac{1}{3}, 0, 0, \ldots \Big)$, then $x^* = \Big( 1, \frac{1}{2}, \frac{1}{3}, 0, 0, \ldots \Big)$.
On the other hand, if $x = \Big( 1, 0, \frac{1}{2}, 0, \frac{1}{3}, 0, \frac{1}{4}, 0, \ldots \Big)$, then $x^* = \Big( 1, \frac{1}{2}, \frac{1}{3}, \frac{1}{4}, \ldots \Big)$.

\section*{Acknowledgment}

We wish to record our appreciation to Roman Sznajder and Michael Orlitzky for their comments and suggestions on an earlier draft of the paper. The work of J.~Jeong was supported by the National Research Foundation of Korea NRF-2016R1A5A1008055 and the National Research Foundation of Korea NRF-2021R1C1C2008350. 

%%%%%%%%%%%%%%%%%%%%%%%%%%%%%%%%%%%%%%%%%%%%%%%%%%%%%%%%%%%%%%%%%%%%%%%%%%


\begin{thebibliography}{}
	\bibitem{alberti-uhlmann} 
	P.M. Alberti and A. Uhlmann, 
	\emph{Stochasticity and Partial Order}, 
	Dt. Verlag Wissenshaften, Berlin, 1982.
	
	\bibitem{ando} 
	T. Ando, 
	\emph{Majorization, doubly stochastic matrices, and comparison of eigenvalues}, 
	Linear Algebra Appl., 118 (1989) 163-248. 
	
	\bibitem{baes}
	M. Baes,
	\emph{Convexity and differentiability properties of spectral functions in Euclidean Jordan algebras},
	Linear Algebra Appl., 422 (2007) 664-700.
	
	\bibitem{bennett-sharpley}
	C. Bennett, R.C. Sharpley,
	\emph{Interpolation of operators}, 
	Academic Press, 1988.
	
	\bibitem{bhatia}
	R. Bhatia,
	\emph{Matrix Analysis},	
	Springer, New York, 1997.
	
	\bibitem{bauschke et al}
	H.H. Bauschke, O. G\"{u}ler, A.S. Lewis, and H.S. Sendov,
	\emph{Hyperbolic polynomials and convex analysis}, 
	Canadian J. Math., 53 (2001) 470-488.
	
	%\bibitem{facchinei-pang}
	%F. Facchinei and J.-S. Pang,
	%\emph{Finite Dimensional Variational Inequalities and %Complementarity Problems, Volumes I \& II},
	%Springer, New York, 2003.
	
	\bibitem{eaton1}
	M.L. Eaton,
	\emph{On group induced orderings, monotone functions, and convolution theorems, in: Inequalities in Statistics and Probability, Y. L. Tong (ed.)}, 
	IMS Lectures Notes, Monograph Series 5, IMS, Hayward (1984) 13-25.
	
	\bibitem{eaton2}
	M. L. Eaton,
	\emph{Group induced orderings with some applications in statistics}, 
	CWI Newsletter, 16 (1987) 3-31.
	
	\bibitem{eaton-perlman}
	M.L. Eaton and M.D. Perlman,
	\emph{Reflection groups, generalized Schur functions, and the geometry of majorization},
	Annals of Probability, 5 (1977) 829-860.
	
	\bibitem{fan}
	K. Fan,
	\emph{On a theorem of Weyl concerning eigenvalues of linear transformations I},
	Proceedings of National Academy of Sciences, USA, 35 (1949) 652-655.
	
	\bibitem{faraut-koranyi}
	J. Faraut and A. Kor\'{a}nyi,
	\emph{Analysis on Symmetric Cones},
	Clarendon Press, Oxford, 1994.
	
	\bibitem{garding}
	L. G\aa rding,
	\emph{An inequality for hyperbolic polynomials},
	J. Math. Mech. 8 (1959) 957-965.
	
	\bibitem{gowda-doubly-stochastic}
	M.S. Gowda, 
	\emph{Positive and doubly stochastic maps, and majorization in Euclidean Jordan algebras}, 
	Linear Algebra Appl., 528 (2017) 40-61.
	
	\bibitem{gowda-ftvn}
	M.S. Gowda,
	\emph{Optimizing certain combinations of linear/distance functions over spectral sets},
	arXiv:1902.06640v2, March 2019.
	
	\bibitem{gowda-ftvn-commutation}
	M.S. Gowda, 
	\emph{Commutation principles for optimization problems on spectral sets in Euclidean Jordan algebras},
	arXiv:2009.04874v1, Sept. 2020, Optimization Letters, to appear.
	
	\bibitem{gowda-jeong}
	M.S. Gowda and J. Jeong,
	\emph{Commutation principles in Euclidean Jordan algebras and normal decomposition principles},
	SIAM J. on Optimization, 27 (2017) 1390-1402.
	
	%\bibitem{gowda-spectral connectedness}
	%M.S. Gowda and J. Jeong,
	%\emph{On the connectedness of spectral sets and %irreducibility of spectral cones in Euclidean Jordan %algebras},
	%Linear Algebra Appl., 559 (2018) 181-193.
	
	\bibitem{gowda-tao} 
	M.S. Gowda and J. Tao,
	\emph{The Cauchy interlacing theorem in simple Euclidean Jordan algebras and some consequences},
	Linear and Multilinear Algebra, 59 (2011) 65-86.
	
	\bibitem{gurvits} 
	L. Gurvits, 
	\emph{Combinatorics hidden in hyperbolic polynomials and related topics},
	arXiv:math/0402088v1 [math.CO], 2004.
	
	\bibitem{jeong-gowda-spectral cone}
	J. Jeong and M.S. Gowda,
	\emph{Spectral cones in Euclidean Jordan algebras},
	Linear Algebra Appl., 509 (2016) 286-305.
	
	\bibitem{jeong-gowda-spectral set}
	J. Jeong and M.S. Gowda,
	\emph{Spectral sets and functions in Euclidean Jordan algebras},
	Linear Algebra Appl., 518 (2017) 31-56.
	
	\bibitem{lewis}
	A. S. Lewis,
	\emph{Group invariance and convex matrix analysis},
	SIAM J. Matrix Anal., 17 (1996) 927-949.
	
	\bibitem{lewis2}
	A. S. Lewis,
	\emph{Convex analysis on the Hermitian matrices},
	SIAM J. Optim., 6 (1996) 164-177.
	
	\bibitem{lewis3}
	A.S. Lewis,
	\emph{Convex analysis on Cartan subspaces},
	Nonlinear Analysis, 42 (2000) 813-820.
	
	\bibitem{lewis et al}
	A. S. Lewis, P.A. Parrilo, and M. V. Ramana, \emph{The Lax conjecture is true},
	Proc. Amer. Math. Soc., 133 (2005) 2495-2499.
	
	\bibitem{lim et al}
	Y. Lim, J. Kim, and L. Faybusovich,
	\emph{Simultaneous diagonalization on simple Euclidean Jordan algebras and its applications},
	Forum Math., 15 (2003) 639-644.
	
	\bibitem{marshall-olkin}
	A.W. Marshall and I. Olkin,
	\emph{Inequalities: Theory of Majorization and its Applications},
	Academic Press, New York (1979).
	
	\bibitem{mirsky}
	L. Mirsky, 
	\emph{Trace of matrix products}, 
	Math. Nach., 20 (1959) 171-174.
	
	%\bibitem{megginson}
	%R.E. Megginson,
	%\emph{An Introduction to Banach Space Theory},
	%Springer, New York (1998).
	
	\bibitem{neumann}
	J. von Neumann, 
	\emph{Some matrix inequalities and metrization of matric-space}, 
	Tomsk University Rev., 1 (1937) 286-300; In Collected Works, Vol. IV, Pergamon, Oxford, 1962, 205-218.
	
	\bibitem{niezgoda-group majorization}
	M. Niezgoda,
	\emph{Group majorization and Schur type inequalities},
	Linear Algebra Appl., 268 (1998) 9-30.
	
	\bibitem{niezgoda-commutation}
	M. Niezgoda,
	\emph{Extended commutation principles for normal decomposition systems},
	Linear Algebra Appl., 539 (2018) 251-273.
	
	\bibitem{orlitzky-note}
	M. Orlitzky,
	\emph{The NDS-FTvN connection},
	private communication, June 20, 2019.
	
	\bibitem{orlitzky-proscribed}
	M. Orlitzky,
	\emph{Proscribed normal decomposition systems of Euclidean Jordan algebras},
	Optimization Online, June 30, 2020.
	
	%\bibitem{ramirez et al}
	%H. Ram\'{i}rez, A. Seeger, and D. Sossa,
	%\emph{Commutation principle for variational problems on %Euclidean Jordan algebras},
	%SIAM J. Optim., 23 (2013) 687-694.
	
	\bibitem{richter}
	H. Richter, 
	\emph{Zur absch\"{a}tsung von matrizennormen}, 
	Math. Nach., 18 (1958) 178-187.
	
	\bibitem{rockafellar}
	R.T. Rockafellar,
	\emph{Convex Analysis},
	Princeton University Press, Princeton, 1970.
	
	\bibitem{rudin} 
	W. Rudin, 
	\emph{Functional Analysis}, 
	McGraw-Hill, New York, 1973.
	
	\bibitem{tam}
	T.-Y. Tam,
	\emph{An extension of a result of Lewis},
	Elec. J. Linear Algebra, 5 (1999) 1-10.
	
	\bibitem{theobald}
	C. M. Theobald,
	\emph{An inequality for the trace of the product of two symmetric matrices},
	Math. Proc. Camb. Philos. Soc., 77 (1975) 265-267.
	
\end{thebibliography}
\end{document}